\documentclass{article}
\usepackage{dutchcal}

\newcommand{\mr}[1]{\mathring{#1}}
\usepackage{graphics}
  \usepackage{amssymb}
    \usepackage{amsthm}
  \usepackage{amsmath}
  \usepackage{pgf}
\usepackage{tikz}
\usetikzlibrary{arrows,automata}
\usepackage{circuitikz}
\usepackage{epstopdf}
  
\def\p{\partial}

\newtheorem{theorem}{Theorem}[section]
\newtheorem{lemma}[theorem]{Lemma}

\newtheorem{proposition}[theorem]{Proposition}
\theoremstyle{definition}
\newtheorem{definition}[theorem]{Definition}
\newtheorem{example}[theorem]{Example}
\newtheorem{remark}[theorem]{Remark}

\newcommand {\A}{\mathbb{A}}
\newcommand {\bay}{\begin{array}}
\newcommand {\eay}{\end{array}}
\newcommand {\bdm}{\begin{displaymath}}
\newcommand {\edm}{\end{displaymath}}
\newcommand{\ups}{\upsilon}
\newcommand{\w}{\varpi}

\newcommand{\vhi}{\varphi}

  \makeatletter
  \let\c@equation\undefined
  \let\c@section\undefined
  \let\c@subsection\undefined
  \let\c@zad\undefined
  \makeatother
  \newcounter{section}

  \newcounter{equation}[section]
  
  \newcounter{subsection}[section]

\newcommand{\beq}{\begin{equation}}
\newcommand{\eeq}{\end{equation}}

\newcommand{\bd}{\begin{displaymath}}
\newcommand{\ed}{\end{displaymath}}
\newcommand{\be}{\begin{equation}}
\newcommand{\ee}{\end{equation}}
\newcommand{\bdow}{\begin{proof}}
\newcommand{\edow}{\end{proof}}

\newcommand{\papap}{\papa \end{proof}}

\newfont{\smoldita}{cmmib8}
\newfont{\boldita}{cmmib10}
\newfont{\bboldita}{cmmib10}

\newcommand{\ti}[1]{\tilde{#1}}

\newcommand{\la}{\lambda}

\newcommand{\mb}[1]{\boldsymbol{#1}}
\newcommand{\mbb}[1]{\mathbb{#1}}
\newcommand{\mc}[1]{\mathcal{#1}}

\newcommand{\me}{\mb{\varepsilon}}
\newcommand{\msf}[1]{\mathsf{#1}}
\newcommand{\spp}{\text{supp}\;}
\newcommand{\bv}{\boldsymbol{v}}

\begin{document}
\begin{center}{\Large Telegraph type systems on networks and port-Hamiltonians. II. Graph realizability}\end{center}

\begin{center}{J. Banasiak\footnote{The research has been partially supported by the National
Science Centre of Poland Grant 2017/25/B/ST1/00051 and the National Research Foundation of South Africa Grant 82770} \\ \small{Department of Mathematics and Applied Mathematics, University of Pretoria}\\ \small{Institute of Mathematics,  \L\'{o}d\'{z} University of Technology}\\ \small{International Scientific Laboratory of
Applied Semigroup Research, South Ural
State University}\\ \small{e-mail: jacek.banasiak@up.ac.za}\\\& \\A. B\l och\footnote{The research was completed while the author was a Doctoral Candidate in the Interdisciplinary Doctoral School at the \L\'{o}d\'{z} University of Technology, Poland.}\\
 \small{Institute of Mathematics,  \L\'{o}d\'{z} University of Technology} \\\small{e-mail: adam.bloch@dokt.p.lodz.pl}}\end{center}
\begin{abstract}
Hyperbolic systems on networks often can be written as systems of first order equations on an interval, coupled by transmission conditions at the endpoints, also called port-Hamiltonians. However, general results for the latter have been difficult to interpret in the network language.
The aim of this paper is to derive conditions under which a port-Hamiltonian with general linear Kirchhoff's boundary conditions  can be written as a system of $2\times 2$ hyperbolic equations on a metric graph $\Gamma$. This is achieved by interpreting the matrix of the boundary conditions as a potential map of vertex connections of $\Gamma$ and then showing that, under the derived assumptions, that matrix can be used to determine the adjacency matrix of $\Gamma$.
\end{abstract}
 \section{Introduction}

 In this paper we are concerned with systems of linear hyperbolic equations on a bounded interval, say, $(0,1)$, sometimes referred to as port-Hamiltonians, \cite{JaMoZw} and the role is played by the boundary conditions coupling the incoming and outgoing Riemann invariants determined by the system at the endpoints at $x=0$ and $x=1$,
 \begin{subequations}\label{syst1}
 \begin{equation}
 \p_t\binom{\mb{\ups}}{\mb{\w}} = \left(\begin{array}{cc}-\mc C_+&0\\0&\mc C_-\end{array}\right)\binom{\mb{\ups}}{\mb{\w}},\quad 0<x<1, t>0,\label{syst1a}
 \end{equation}
 \begin{equation}
 \mb \Xi (\mb{\ups}(0,t),\mb{\w}(1,t),\mb{\ups}(1,t),\mb{\w}(0,t))^T=0,\quad t>0,\label{syst1b}
 \end{equation}
 \begin{equation}
 \mb{\ups}(x,0)= \mr{\mb\ups}(x),\; \mb{\w}(x,0)= \mr{\mb\w}(x)\quad 0<x<1,\label{syst1c}
 \end{equation}
 \end{subequations}
  where $\mb\ups$ and $\mb\w$ are the Riemann invariants flowing, respectively, from $0$ to $1$ and from $1$ to $0$, $\mc C_+$ and $\mc C_-$ are $m_+\times m_+$ and  $m_-\times m_-$ diagonal matrices with positive entries and, with $2m=m_++m_-$, $\mb \Xi$ is an $2m\times 4m$ matrix relating outgoing  $\mb{\ups}(0),\mb{\w}(1)$ and incoming $\mb{\ups}(1),\mb{\w}(0)$ boundary values so that \eqref{syst1b} can be written as
  \begin{equation}
  \mb \Xi_{out} (\mb{\ups}(0,t),\mb{\w}(1,t))^T + \mb\Xi_{in}(\mb{\ups}(1,t),\mb{\w}(0,t))^T=0,\quad t>0.\label{syst1b1}
 \end{equation}
   An important class of such problems arises from dynamical systems on metric graphs.  Let $\Gamma$ be a graph with $r$ vertices $\{\bv_j\}_{1\leq j\leq r}=:\Upsilon$ and $m$ edges $\{\mb e_j\}_{1\leq j\leq m}$ (identified with $(0,1)$ through a suitable parametrization).  The dynamics on each edge $\mb e_j$ is described by
 \begin{equation}
 \p_t \mb p^j+ \mc M^j\p_x\mb p^j  =0, \quad t>0, 0<x<1, 1\leq j\leq m,
 \label{sys1}
 \end{equation}
 where $\mb p^j = (p^j_1, p^j_2)^T$ and $\mc M^j = (M^j_{lk})_{1\leq k,l\leq 2}$ are defined on $[0,1]$ and $\mc M^j(x)$ is a strictly hyperbolic real matrix for each $x\in [0,1]$ and $1\leq j\leq m$. System \eqref{sys1} is complemented with initial conditions and  suitable transmission conditions coupling the values of $\mb p^j$ at the vertices which the edges $\mb e_j$ are incident to. Then \eqref{syst1} can be obtained from \eqref{sys1} by diagonalization so that (suitably re-indexed) $\mb \ups$ and $\mb \w$ are the Riemann invariants of $\mb p = (\mb p^j)_{1\leq j\leq m}$.

 Such problems have been a subject of intensive research, both from the dynamics on graphs,  \cite{AM2, DKNR, BN, BFN3, KMN, Kuch, DM}, and the 1-D hyperbolic systems, \cite{BaCor, Zwart2010, JaZwbook, JaMoZw}, points of view.  However, there is hardly any overlap,  as there seems to be little interest in the network interpretation of the results in the latter, while in the former the conditions on the Riemann invariants seems to be ''difficult to adapt to the case of a network'', \cite[Section 3]{KMN}.

 The main aim of this paper, as well as of the preceding one, \cite{JBAB1}, is to bring together these two approaches.  In \cite{JBAB1} we have provided explicit formulae allowing for a systematic conversion of Kirchhoff's type network transmission conditions to \eqref{syst1b} in such a way that the resulting system \eqref{syst1} is well-posed. We also gave a proof of the well-posedness in any $L_p$, $1\leq p<\infty$, which, in contrast to \cite{Zwart2010}, is purely semigroup theoretic. For notational clarity, we focused on  $2\times 2$ hyperbolic systems on each edge but the method works equally well for systems of arbitrary (finite) dimension.   In this paper we are concerned with the reverse question, that is, to determine under what assumptions on $\mb \Xi$,  \eqref{syst1} describes a network dynamics given by $2\times 2$ hyperbolic systems on each edge, coupled by Kirchhoff's transmission conditions at incident vertices.

To briefly describe the content of the paper, we observe that if the matrix $\mb\Xi = \{\xi_{ij}\}_{1\leq i\leq 2m, 1\leq j\leq 4m}$ in \eqref{syst1b} describes transmission conditions at the vertices of a graph, say $\Gamma,$ on whose edges we have $2\times 2$ systems of hyperbolic equations, then we should be able to group the indices $j$ into pairs $\{j',j''\}$ corresponding to the edges of $\Gamma$ on which we have  $2\times 2$ systems for the components $j'$ and $j''$ of $(\mb\ups,\mb\w).$  Thus, in a sense, the columns of $\mb\Xi$ determine the edges of $\Gamma.$ It follows that it is easier to split the reconstruction of $\Gamma$ into two steps and first build a digraph $\mb \Gamma,$ where each column index $j$ of $\mb\Xi$ is associated with an arc, say $\me^j,$ on which we have a first order system for either $\ups_j$ or $\w_j.$  Thus, the main problem is to construct vertices of $\Gamma$ (and $\mb \Gamma$) which should be somehow determined by a partition of the row indices of $\mb\Xi$. To do this, we observe that the coefficients of $\mb\Xi$ represent a map of connections of the edges in the sense that, roughly speaking, if $\xi_{ij}\neq 0$ and $\xi_{ik}\neq 0$, then arcs $\me^j$ and $\me^k$ are incident to the same vertex and, if they are incoming to it, then they cannot be incoming to any other vertex. A difficulty here is that  while for the flow to occur from, say, $\me^j$ to $\me^k,$ these arcs must be incident to the same vertex but the converse may not hold, that is, for $\me^j$ incoming to and $\me^k$ outgoing from the same $\mb v,$ the flow from $\me^j$ may not enter $\me^k$ but go to other outgoing arcs. To avoid such a case, in this paper we  formulate conditions ensuring that the flow connectivity at each vertex is the same as the graph connectivity. This assumption yields a relatively simple criterion for the reconstruction of $\mb \Gamma,$ which is that $\widehat {\left(\widehat{\mb \Xi_{out}}\right)^T\widehat{\mb \Xi_{in}}}$ is the adjacency matrix of a line graph (where for a matrix $\mc A$, $\widehat{\mc A}$ is obtained by replacing non-zero entries of $\mc A$ by 1.) This, together with some technical assumptions, allows us to apply the theory of \cite{HemBein}, see also \cite[Theorem 4.5.1]{Bang},  to construct first $\mb \Gamma$ and then $\Gamma$ in such a way that \eqref{syst1b1} can be localized at each vertex of $\Gamma$ in a way which is consistent with \eqref{syst1a}.

The main idea of this paper is similar to that of \cite{BF1}.  However, \cite{BF1} dealt with first order problems with \eqref{syst1b1} solved with respect to the outgoing data. Here, we do not make this assumption and, while \eqref{syst1} technically is one-dimensional, having reconstructed $\mb \Gamma,$ we still have to glue together its  pairs of arcs to obtain the edges of $\Gamma$ in such a way that the corresponding pairs of solutions of \eqref{syst1a} are Riemann invariants of $2\times 2$ systems on $\Gamma$. Another difficulty in the current setting is  potential presence of sources and sinks in $\Gamma.$ Their structure is not reflected in the line graph, \cite{BF1}, and reconstructing them in a way consistent with a system of  $2\times 2$ equations on $\Gamma$ is technically involved.

The paper is organized as follows. In Section 2 we briefly recall the notation and relevant results from \cite{JBAB1}. Section 3 contains the main result of the paper. In Appendix we recall basic results on line graphs in the interpretation suitable for the considerations of the paper.

 \section{Notation, definitions and earlier results}

   We consider a network represented by a finite, connected and simple (without loops and multiple edges) metric graph $\Gamma$ with $r$ vertices $\{\bv_j\}_{1\leq j\leq r}=:\Upsilon$ and $m$ edges $\{\mb e_j\}_{1\leq j\leq m}$. We denote by $E_{\bv}$ the set of edges incident to $\bv,$ let $J_{\bv} =\{j;\; \mb e_j \in E_{\bv}\} $  and  $|E_{\bv}| =|J_{\bv}|$ be the valency of $\bv$. We identify the edges with unit intervals through sufficiently smooth invertible functions $l_j: \mb e_j \mapsto [0,1]$. In particular, we call $\bv$ with $l_j(\bv) =0$ the tail of $\mb e_j$ and the head if $l_j(\bv)=1$.  On each edge $\mb e_j$ we consider system \eqref{sys1}. Let $\la^j_-<\la^j_+$ be the eigenvalues of $\mc M^j, 1\leq j\leq m$ (the strict inequality is justified by the strict hyperbolicity of $\mc M^j$).  The eigenvalues can be of the same sign as well as of different signs. In the latter case, we have $\la^j_-<0<\la^j_+$. By $ f^j_\pm = (f^j_{\pm,1}, f^j_{\pm,2})^T$ we denote the eigenvectors corresponding to, respectively, $\la^j_\pm$ and by $$
 \mc F^j = \left(\begin{array}{cc} f^j_{+,1}&f^j_{-,1}\\
   f^j_{+,2}&f^j_{-,2} \end{array}\right),
   $$
   the diagonalizing matrix on each edge. The Riemann invariants $\mb u^j = (u^j_1,u^j_2)^T, 1\leq j\leq m,$ are defined by
 \begin{equation}
 \mb u^j = (\mc F^{j})^{-1} \mb p^j\quad\text{and}\quad
 \mb p^j = \binom{f^j_{+,1} u^j_1 + f^j_{-,1}u^j_2}{f^j_{+,2}u^j_1 + f^j_{-,2}u^j_2}.\label{pguw}
 \end{equation}
 Then we diagonalize \eqref{sys1} and, discarding lower order terms, we  consider
  \begin{equation}
 \p_t\mb u^j =  \mc L^j\p_x\mb{u}^j = \left(\begin{array}{cc}-\la^j_+&0\\0&-\la^j_- \end{array}\right)\p_x\mb{u}^j,\label{sysdiag}
 \end{equation}
 for each $1\leq j\leq m$.

\subsection{The boundary conditions}

The most general linear local boundary conditions at $\mb v\in \Upsilon$ are given by
\begin{equation}
  \mb \Phi_{\bv} \mb p(\bv) = 0,
  \label{bc2s}
  \end{equation}
where $\mb p = ((p_1^j,p_2^j)_{1\leq j\leq m})^T$ and the real matrix $\mb \Phi_{\bv}$ is given by
\begin{equation}
\mb \Phi_{\bv} := \left(\begin{array}{ccccc}\phi^{j_1}_{\bv,1}&\vhi^{j_1}_{\bv,1}&\ldots&\phi^{j_{|J_{\bv}|}}_{\bv,1}&\vhi^{j_{|J_{\bv}|}}_{\bv,1}\\
\vdots&\vdots&\vdots&\vdots&\vdots\\
\phi^{j_1}_{\bv,k_{\bv}}&\vhi^{j_1}_{\bv,k_{\bv}}&\ldots&\phi^{j_{|J_{\bv}|}}_{\bv,k_{\bv}}&\vhi^{j_{|J_{\bv}|}}_{\bv,k_{\bv}}\end{array}\right),\
\label{Phiv}
\end{equation}
where $J_{\bv} = \{j_1, \ldots, j_{|J_{\bv}|}\}$ and $k_{\mb v}$ is a parameter determined by the problem. The difficulty with such a formulation is that it is not immediately clear what properties $\mb \Phi_{\bv}$ should have to ensure well-posedness of the hyperbolic problem for which \eqref{bc2s}, $\bv\in \Upsilon$, serve as boundary conditions. There are various ways around this difficulty. In e.g. \cite{Nic, KMN},  conditions are imposed directly on $\mb \Phi_v$ to ensure specific properties, such as dissipativity, of the resulting initial boundary value problem. We, however, follow the paradigm introduced in \cite[Section 1.1.5.1]{BaCor} and require that at each vertex all outgoing data must be determined by the incoming data. Since for a general system  \eqref{sys1} it is not always obvious which data are outgoing and which are incoming at a vertex, we write \eqref{bc2s} in the equivalent form using the Riemann invariants $\mb u = \mc F^{-1} \mb p$, as
\begin{equation}
\mb \Psi_{\bv} \mb u(\bv) : = \mb \Phi_{\bv} \mc F(\bv)\mb u(\bv)  = 0.
  \label{bc2uw}
  \end{equation}
For Riemann invariants, we can define their outgoing values at $\bv$ as follows.
\begin{definition} Let $\mb v\in \Upsilon$. The following values $u^j_k(\mb v), j\in J_{\bv}, k=1,2,$ are outgoing at $\bv,$
 \begin{center}
 \begin{tabular} {|c|c|c|c|}
 \hline
 If &$\la^j_+>\la^j_->0$& $\la^j_+>0>\la^j_-$&$0>\la^j_+>\la^j_-$\\
 \hline
 $l_j(\bv) =0$& $u^j_1(\bv), u^j_2(\bv)$ & $u^j_1(\bv)$& none\\
 \hline
 $l_j(\bv) =1$&none &$u^j_2(\bv)$&$u^j_1(\bv), u^j_2(\bv)$ \\
 \hline
 \end{tabular}\;.
 \end{center}
  \label{tab1}
  \end{definition}
Denote by $\alpha_j$ the number of positive eigenvalues on $\mb e_j$. Then we see that for a given vertex $\bv$ with valence $|J_{\bv}|$ the number of outgoing values is given by
 \begin{equation}
 k_{\bv} := \sum\limits_{j\in J_{\bv}} (2(1-\alpha_j)l_j(\bv) +\alpha_j ).
 \label{kv1}
 \end{equation}

  \begin{definition} We say that $\bv$ is a \textit{sink}, and write $\bv\in \Upsilon_z$, if either $\alpha_j =2$ and $l_j(\bv) =1$ or $\alpha_j =0$ and $l_j(\bv) =0$ for all $j \in J_{\bv}$.  We say that $\bv$ is a \textit{source}, and write $\bv\in \Upsilon_s$, if either $\alpha_j =0$ and $l_j(\bv) =1$ or $\alpha_j =2$ and $l_j(\bv) =0$ for all $j \in J_{\bv}$. If $\bv$ is neither a source nor a sink, then we say that $\bv$ is a \textit{transient} (or \textit{internal}) vertex and write $\bv \in \Upsilon_t$.
\end{definition}
We observe that if $\bv\in \Upsilon_z$, then  $k_{\bv}=0$ (so that no boundary conditions are imposed at a sink), while  if $\bv\in \Upsilon_s$, then  $k_{\bv}=2|J_{\bv}|$.

A typical example of \eqref{bc2uw} is Kirchhoff's law that requires that the total inflow rate into a vertex must equal the total outflow rate from it. Since it provides only one equation, in general it is not sufficient to ensure the well-posedness of the problem. So, we introduce the following definition.
\begin{definition}
We say that $\mb p$ satisfies a generalized Kirchhoff conditions at $\bv \in \Upsilon\setminus\Upsilon_z$ if, for $\mb u = \mc F^{-1} \mb p,$
  \eqref{bc2uw} is satisfied for some matrix $\mb\Phi_v= \mb\Psi_v\mc F^{-1}$ with $k_{\bv}$ given by \eqref{kv1}.
  \end{definition}

 To realize the requirement that the outgoing values should be determined by the incoming ones, we have to analyze the structure of $\mb\Psi_v$.  Let us  introduce the partition  \begin{equation} \{1,\ldots,m\} = : J_1\cup J_2\cup J_0,\label{Jpart}\end{equation}
 where $j \in J_1$ if $\alpha_j =1, $  $j \in J_2$ if $\alpha_j =2$  and  $j\in J_0$ if $\alpha_j=0$. This partition induces the corresponding partition of each $J_{\bv}$ as
  $$J_{\bv} := J_{\bv,1}\cup J_{\bv,2}\cup J_{\bv,0}.
  $$
 We also consider another partition $J_{\bv} = J_{\bv}^0\cup J_{\bv}^1,$ where $j\in J_{\bv}^0$ if $l_j(\bv) =0$ and $j\in J_{\bv}^1$ if $l_j(\bv) =1$. Then we can give an alternative expression for $k_{\bv}$ as
 \begin{equation}
 k_{\bv} = \sum\limits_{j\in J^0_{\bv}} \alpha_j + \sum\limits_{j\in J^1_{\bv}} (2-\alpha_j) = |J_{\bv,1}|+ 2(|J^0_{\bv}\cap J_{\bv,2}| + |J^1_{\bv}\cap J_{\bv,0}|).
 \label{kval}
 \end{equation}
 Then, by \cite[Lemma 3.6]{JBAB1},
   \begin{description}
             \item (i) $u^j_1(0)$ is outgoing if and only if $j \in (J_{\bv,1}\cup J_{\bv,2})\cap J^0_{\bv},$
               \item (ii) $u^j_2(0)$ is outgoing if and only if $j \in J_{\bv,2}\cap J^0_{\bv},$
               \item (iii) $u^j_1(1)$ is outgoing if and only if $j \in J_{\bv,0}\cap J^1_{\bv},$
               \item (iv) $u^j_2(1)$ is outgoing if and only if $j \in (J_{\bv,1}\cup J_{\bv,0})\cap J^1_{\bv}.$
  \end{description}
     We introduce the block diagonal matrix
        \begin{equation}
    \mc {\ti {F}}_{out}(\bv) = diag\{\mc {\ti {F}}_{out}^j(\bv)\}_{j \in J_{\bv}},
    \label{frakVj}
    \end{equation}
    where
    $$
    \mc {\ti F}_{out}^j(\bv) = \left\{\begin{array} {ccc} \left(\begin{array}{cc}0&0\\0&0\end{array}\right)&\text{if} &j\in (J_{\bv,0} \cap J_{\bv}^0)\cup (J_{\bv,2}\cap J_{\bv}^1),\\
    \left(\begin{array}{cc}f^j_{+,1}(l_j(\bv))&f^j_{-,1}(l_j(\bv))\\f^j_{+,2}(l_j(\bv))&f^j_{-,2}(l_j(\bv))\end{array}\right)&\text{if}& j\in (J_{\bv,0} \cap J_{\bv}^1)\cup (J_{\bv,2}\cap J_{\bv}^0),\\
    \left(\begin{array}{cc}f^j_{+,1}(0)&0\\f^j_{+,2}(0)&0\end{array}\right)&\text{if} &j\in J_{\bv,1} \cap J_{\bv}^0,\\
    \left(\begin{array}{cc}0&f^j_{-,1}(1)\\0&f^j_{-,2}(1)\end{array}\right)&\text{if}& j\in J_{\bv,1} \cap J_{\bv}^1.
    \end{array}
    \right.
        $$
    Further, by $\mc F_{out}(\bv)$ we denote the contraction of $\mc {\ti F}_{out}(\bv)$; that is, the $2|J_{\bv}| \times k_{\bv}$ matrix obtained from $\mc {\ti F}_{out}(\bv)$ by deleting $2|J_{\bv}| - k_{\bv}$ zero columns, and then define $\mc F_{in}(\bv)$ as the analogous contraction of $\mc F(\bv)-\mc{\ti F}_{out}(\bv)$.

        In a similar way, we extract from $ \mb u(\bv)$ the outgoing boundary values $\widetilde{\mb u}_{out}(\bv) =(\widetilde{\mb u}^j_{out}(\bv))_{j\in J_{\mb v}}$ by
    $$
        \widetilde{\mb u}^j_{out}(\bv) = \left\{\begin{array} {ccc} (0,0)^T&\text{if} &j\in (J_{\bv,0} \cap J_{\bv}^0)\cup (J_{\bv,2}\cap J_{\bv}^1),\\
   (u^j_1(l_j(\bv)),u^j_2(l_j(\bv)))^T&\text{if}& j\in (J_{\bv,0} \cap J_{\bv}^1)\cup (J_{\bv,2}\cap J_{\bv}^0),\\
    (u^j_1(0),0)^T&\text{if} &j\in J_{\bv,1} \cap J_{\bv}^0,\\
    (0,u^j_2(1))^T&\text{if}& j\in J_{\bv,1} \cap J_{\bv}^1,
    \end{array}
    \right.
    $$
        and $        \widetilde{\mb u}_{in}(\bv) = {\mb u(\bv)}-\widetilde{\mb u}_{out}(\bv).
        $
As above, we define $\mb u_{out}(\bv)$ to be the vector in $\mbb R^{k_{\bv}}$ obtained by discarding the zero entries in $\widetilde{\mb u}_{out}(\bv)$, as described above and, similarly, $\mb u_{in}(\bv)$ is the vector in $\mbb R^{2|J_{\bv}|-k_{\bv}}$ obtained from $\widetilde{\mb u}_{in}(\bv)$.
    \begin{proposition}\cite[Proposition 3.8]{JBAB1}
     Boundary system \eqref{bc2uw} at $\bv\in \Upsilon\setminus \Upsilon_z$ is equivalent to
    \begin{equation}
 \mb \Phi_{\bv} \mc F_{out}(\bv) \mb u_{out}(\bv)
    + \mb \Phi_{\bv} \mc F_{in}(\bv)\mb u_{in}(\bv) = 0
    \label{bcsplit1}
    \end{equation}
    and hence it  uniquely determines the outgoing values of $\mb u(\bv)$ at $\bv$ as defined by Definition \ref{tab1} if and only if
    \begin{equation}
    \mb\Phi_{\bv} \mc F_{out}(\bv)\quad \text{is\; nonsingular}.
    \label{compsolv}
    \end{equation}
     Then
    \begin{equation} \mb u_{out}(\bv) = - (\mb \Phi_{\bv} \mc F_{out}(\bv))^{-1} \mb \Phi_{\bv} \mc F_{in}(\bv)\mb u_{in}(\bv).
    \label{mcB}
    \end{equation}
     \label{prop1}
  \end{proposition}
To pass from \eqref{sys1} with Kirchhoff's boundary conditions at each vertex $\bv \in \Upsilon\setminus \Upsilon_z$ to \eqref{syst1} we have to write the former in a global form.  Assuming the vertices in $\Upsilon\setminus\Upsilon_z$ are ordered as $\{\bv_1, \ldots, \bv_{r'}\}$, we  define $\mb \Psi' = diag \{\mb \Psi_{\bv} \}_{\bv\in \Upsilon\setminus\Upsilon_z}$ and $\gamma \mb u  = ((\mb u(\bv))_{\bv \in \Upsilon\setminus\Upsilon_z})^T$ and write \eqref{bc2uw} as
  \begin{equation}
  \mb \Psi' \gamma \mb u = 0.
  \label{bc2uwgl}
  \end{equation}
    We note that the function values that are incoming at $\bv\in \Upsilon_z$ do not influence any outgoing data. To keep, however, the track of all vertex values, we augment  $\mb \Psi'$ with zero columns corresponding to edges coming to sinks and denote such an augmented matrix by $\mb\Psi$. Since, by the hand shake lemma, we have $  2\sum_{\bv\in \Upsilon} |J_{\bv}| = 4m$ and, by \cite[Section 3.2]{JBAB1},
    $\sum_{\bv\in \Upsilon\setminus\Upsilon_z} k_{\bv} = 2m,$ $\mb\Psi$ is a $2m\times 4m$ matrix. In the same way, we can provide a global form of \eqref{bcsplit1}, splitting  \eqref{bc2uwgl} as
    \begin{equation}
 \mb \Psi^{out} \gamma \mb u_{out}
    + \mb \Psi^{in}\gamma\mb u_{in} = 0,
    \label{bcsplit2}
    \end{equation}
where $ \mb \Psi^{out} = diag \{\mb \Phi_{\bv} \mc F_{out}(\bv) \}_{\bv\in \Upsilon\setminus\Upsilon_z}$ and
 $\mb \Psi^{in} = diag \{\mb \Phi_{\bv} \mc F_{in}(\bv) \}_{\bv\in \Upsilon\setminus\Upsilon_z},$ augmented by zero columns corresponding to the incoming functions at the sinks, $\gamma \mb u_{out} := ((\mb u_{out}(\bv))_{\bv \in \Upsilon\setminus\Upsilon_z})^T$, and $\gamma \mb u_{in}$ is $((\mb u_{in}(\bv))_{\bv \in \Upsilon\setminus\Upsilon_z})^T$ augmented by incoming values at the sinks.

 Using the adopted parametrization and the formalism of Definition \ref{tab1}, we only need to distinguish between functions describing the flow from $0$ to $1$ and from $1$ to $0$.
 Accordingly, we group the Riemann invariants $\mb u$  into parts corresponding to positive and negative eigenvalues and rename them as:
 \begin{equation}
\begin{split}
\mb {\upsilon}  &:=\left((u^j_1)_{j\in J_1\cup J_2},(u^j_2)_{j\in J_2}\right) = (\ups_j)_{j\in J^+},\\
\mb {\w}& := \left((u^j_1)_{j\in J_0}, (u^j_2)_{j\in J_1\cup J_0}\right) = (\w_j)_{j\in J^-},
\end{split}
\label{renum}
\end{equation}
where $J^+$  and $J^-$  are the sets of indices $j$ with, respectively, at least 1 positive eigenvalue, and at least 1 negative eigenvalue of $\mc M^j$. In  $J^+$ (respectively $J^-$) the indices from $J_2$ (respectively $J_0$) appear twice so that we renumber them in some consistent way to avoid confusion. For instance, we can take  $J^+ =\{1,\ldots, m^u,m^u+1,\ldots, m^+\}$ and $J^- =\{m^++1,\ldots,m_u, m_u+1,\ldots, 2m\}$ and there are bijections between, respectively, $J_1\cup J_2$ and $\{1,\ldots, m^u\}$,
$J_2$ and $\{m^u+1,\ldots, m^+\}$, $J_0$ and $\{m^++1,\ldots,m_u\}$, and $J_1\cup J_0$ and $\{m_u+1,\ldots, 2m\}.$
We emphasize that such a renumbering is largely arbitrary and different ways of doing it result in just re-labelling of the components of \eqref{syst1} without changing its structure.

In this way, we converted $\Gamma$ into a multi digraph $\mb \Gamma$ with the same vertices $\Upsilon$, in such a way that each edge of $\Gamma$ was split into two arcs parametrized by $x\in [0,1],$ where $x=0$ on each arc corresponds  to the same vertex in $\Gamma$ and the same is valid for $x=1$. Conversely, if we have a multi digraph $\mb \Gamma,$ where all edges appear in pairs and  each two edges joining the same vertex are  parametrized concurrently, then we can collapse $\mb{\Gamma}$ to a graph $\Gamma$.

Using this construction,  the second order hyperbolic problem \eqref{sys1}, \eqref{bcsplit2} was transformed into first order system \eqref{syst1} with \eqref{bcsplit2} written in the form \eqref{syst1b1}. It is clear, however, that \eqref{syst1} can be formulated with an arbitrary matrix $\mb\Xi$. Thus, we arrive at the main problem considered in this paper, how to characterize matrices $\mb\Xi$ that arise from $\mb \Psi$ so that \eqref{syst1} describes a network dynamics.

\section{Graph realizability of port-Hamiltonians}
\subsection{Connectivity at a vertex}
For a graph $\Gamma$, let us consider the multi digraph $\mb\Gamma$ constructed above. The sets of vertices $\Upsilon$ are the same for $\Gamma$ and $\mb\Gamma$. For $\bv \in\Upsilon$ of $\mb{\Gamma},$ we can talk about incoming and outgoing arcs which are determined  by $j\in J_{\bv},$ $l_j(\bv)$ and the signs of $\la_+^j$ and $\la_-^j$, as in Definition \ref{tab1}. We denote by $\mb J_{\bv}^+$ and $\mb J_{\bv}^-$ the (ordered) sets of indices of arcs $\me^j$ incoming and, respectively, outgoing from $\bv$ in $\mb{\Gamma}$. We note that $|\mb J_{\bv}^-|=k_{\bv}$, the number of the outgoing conditions.  With this notation, the matrix $\mb{\Psi_{\bv}}$ can be split into two matrices
$$
\mb{\Psi}_{\bv}^{out}=(\psi_{\bv,i}^j)_{1\leq i\leq k_{\bv},j\in \mb J_{\bv}^-}, \qquad \mb{\Psi}_{\bv}^{in}=(\psi_{\bv,i}^j)_{1\leq i\leq k_{\bv},j\in \mb J_{\bv}^+}.
$$
Since no outgoing value should be missing, we  assume
\begin{equation}
\text{no column or row of $\mb\Psi_{\bv}^{out}$ is identically zero.}\label{ass1}
\end{equation}
These matrices provide some insight into how the arcs are connected by the flow which is an additional feature, superimposed on the geometric structure of the incoming and outgoing arcs at the vertex. In principle, these two structures do not have to be the same, that is, it may happen that the substance flowing from $\me^j$,  $j \in \mb J_{\bv}^+,$ is only directed to some of the outgoing arcs. An extreme case of such a situation is when  both $\mb{\Psi}_{\bv}^{out}$ and $\mb{\Psi}_{\bv}^{in}$ are completely decomposable, see \cite{BruHaMi}, with blocks in both matrices having the same row indices. Then, from the flow point of view, $\bv$ can be regarded as several nodes of the flow network, which are not linked with each other.  Such cases, where the geometric structure at a vertex is inconsistent with the flow structure, may generate problems in determining the graph  underlying transport problems. Thus in this paper we adopt assumptions ensuring that  the map of the flow connections given by the matrices $\mb{\Psi}_{\bv}^{out}$ and $\mb{\Psi}_{\bv}^{in}$ coincides with the geometry at $\bv$. We begin with the necessary definitions.

Consider first a transient vertex $\bv$.
We say that the arc $\me^j$, $j\in \mb J^+_{\bv},$ \textit{flow connects} to $\me^l, l \in \mb J^-_{\bv},$  if $\psi_{\bv,i}^j \neq 0$ and $\psi_{\bv,i}^l\neq 0$ for some $1\leq i\leq k_{\bv}.$ Using this idea, we construct a connectivity matrix $\mathsf C_{\bv} = (\msf c_{\bv,lj})_{l \in \mb J^-_{\bv}, j \in \mb J^+_{\bv}},$ where
$$
\msf c_{\bv,lj} = \left\{\begin{array}{lcl} 1&\text{if}& \me^j\;\text{flow\;connects\;to}\;\me^l,\\
0&&\text{otherwise}.
\end{array}
\right.
$$
\begin{remark}
We observe that the above definition implies that  for $\me^j$ and $\me^l$ to be flow connected, $\me^j$ and $\me^l$ must be incident to the same vertex.
\end{remark}
\begin{remark} We observe that $\mathsf C_{\bv}$ can be interpreted as the adjacency matrix of the bipartite line digraph constructed from the incoming and outgoing arcs at $\bv,$ where the connections between arcs are defined by flow connections, see \cite{BruHaMi}. \end{remark}
For an arbitrary matrix $A = (a_{ij})_{1\leq i\leq p, 1\leq j\leq q},$ by $\widehat A = (\hat a_{ij})_{1\leq i\leq p, 1\leq j\leq q}$ we denote the matrix with every nonzero entry of $A$ replaced by 1.
\begin{lemma}\label{lemmbbc1}
If $\bv$ is a transient vertex,
\begin{equation}
\msf C_{\bv} = \widehat{\left(\widehat{\mb{\Psi}^{out}_{\bv}}\right)^T   \widehat{\mb{\Psi}^{in}_{\bv}}}.
\label{mbbc1}
\end{equation}
\end{lemma}
\begin{proof} Denote  $\msf B =\widehat{\left(\widehat{\mb{\Psi}^{out}_{\bv}}\right)^T   \widehat{\mb{\Psi}^{in}_{\bv}}}.$ Then $\msf b_{ij} =1, i\in \mb J_{\bv}^-, j \in \mb J_{\bv}^+$ if and only if
$$
\sum\limits_{r=1}^{k_{\bv}} \hat{\psi}_{\bv,r}^i\hat{\psi}^j_{\bv,r} \neq 0.
$$
This occurs if and only if there is $r=1,\ldots,k_{\bv}$ such that both $\hat{\psi}_{\bv,r}^i\neq 0$ and $\hat{\psi}^j_{\bv,r} \neq 0,$ which is equivalent to $\me^j$  flow connecting with $\me^i$, that is, $\msf c_{\bv,ij} =1$.
\end{proof}
Let $\bv$ be a source (as we do not impose boundary conditions on sinks). As above, we need to ensure that the flow from a source cannot be split into several isolated subflows. Though here we do not have inflows and outflows, we use a similar idea and say that $\me^i$ and $\me^j$, $i,j \in \mb J_{\bv}^-,$ are \textit{flow connected} if there is $l \in\{1,\ldots, k_{\bv}\}$ such that $\psi_{\bv,l}^i\neq 0$ and $\psi_{\bv,l}^j\neq 0.$
As before, we construct a connectivity matrix $\mathsf C_{\bv} = (\msf c_{\bv,ij})_{i,j \in \mb J^-_{\bv}},$ where
\begin{equation}
\msf c_{\bv,ij} = \left\{\begin{array}{lcl} 1&\text{if}& \me^j\;\text{and}\; \me^i\;\text{are\;flow\;connected},\\
0&&\text{otherwise}.
\end{array}
\right.\label{flcons}
\end{equation}
Note that, contrary to the internal vertex, the connectivity matrix is symmetric. We also do not stipulate that $i\neq j$ so that $\me^j$ is always flow connected to itself and hence, by \eqref{ass1}, each entry of the diagonal of  $\msf C_{\bv}$ is 1.   We have similarly
\begin{lemma}
If $\bv$ is a source,
\begin{equation}
\msf C_{\bv} = \widehat{\left(\widehat{\mb{\Psi}^{out}_{\bv}}\right)^T   \widehat{\mb{\Psi}^{out}_{\bv}}}.
\label{mbbc2}
\end{equation}
\label{lemmbbc2}
\end{lemma}
\begin{proof}
As before, let  $\msf B =\widehat{\left(\widehat{\mb{\Psi}^{out}_{\bv}}\right)^T   \widehat{\mb{\Psi}^{out}_{\bv}}}.$ Then $\msf b_{ij} =1, i,j\in \mb J_{\bv}^-,$ if and only if
$$
\sum\limits_{r=1}^{k_{\bv}} \hat{\psi}_{\bv,r}^i\hat{\psi}^j_{\bv,r} \neq 0.
$$
Certainly, by \eqref{ass1}, $\msf b_{ii} =1, i\in \mb J_{\bv}^-.$ For $i\neq j$, this occurs if and only if there is $r\in\{1,\ldots,k_{\bv}\}$ such that both $\hat{\psi}_{\bv,r}^i\neq 0$ and $\hat{\psi}^j_{\bv,r} \neq 0$ which is equivalent to $\me^j$ and $\me^i$ being flow connected, that is, $\msf c_{\bv,ij} =1$.
\end{proof}
We adopt an assumption that the structure of flow connectivity is the same as of the geometry at the vertex. Thus, if $\bv$ is an internal vertex  and $j\in \mb J_{\mb v}^+$ and $i\in \mb J^-_{\bv}$, then  $\me^j$ flow connects to $\me^i$.  Mathematically, we  assume that
\begin{equation}
\msf C_{\bv}= \mb 1_{\bv} = \left(\begin{array}{cccc}1&1&\ldots&1\\
\vdots&\vdots&\vdots&\vdots\\
1&1&\ldots&1
\end{array}\right)
\label{sfC}
\end{equation}
for all $\bv \in \Upsilon_t;$ the dimension of $\msf C_{\bv}$ is $|\mb J_{\bv}^-|\times |\mb J^+_{\bv}|$.

If $\bv \in \Upsilon_s$, then we assume that the outflow from $\bv$ cannot be separated into independent subflows, that is,  that the arcs outgoing from $\bv$ cannot be divided into groups such that no arc in any group is flow connected to an arc in any other. Equivalently, for each two arcs $\me^i$ and $\me^j$, $i,j \in \mb J_{\bv}^-,$ there is a sequence $j=j_0, j_1,\ldots, j_k = i$ such that $\me^{j_r}$ and $\me^{j_{r+1}}$, $r=0,\ldots,k-1$ are flow connected. Indeed, if such a division was possible, then it would be impossible to find such a sequence between indices $j$ and $i$ in different groups as some pair would have to connect arcs from these different groups. Conversely, if for some arcs $\me^i$ and $\me^j$ there is no such a sequence, then we can build two groups of indices containing, respectively, $i$ and $j$, by considering all indices for which such sequences can be found. Clearly, no arc in the first group is flow connected to any arc in the second as otherwise there would be a sequence connecting $\me^j$ and $\me^i$. Since, by Lemma \ref{lemmbbc2}, $\msf C_{\bv}$ can be considered as the adjacency matrix of the graph with vertices given by $\{\me^j\}_{j \in \mb J^-_{\bv}}$ and the edges determined by the flow connectivity \eqref{flcons}, and $\msf C_{\bv}$ is symmetric, we see that the above assumption is equivalent to
\begin{equation}
\msf C_{\bv} \;\text{is\;irreducible}.
\label{assirr}
\end{equation}
\begin{remark}
Assumption \eqref{assirr} is weaker than requiring that each two arcs from $\{\me^j\}_{j \in \mb J^-_{\bv}}$ are flow connected. Then we would have $\msf C_{\bv}= \mb 1_{|\mb J_{\bv}^-|\times |\mb J^-_{\bv}|}.$
\end{remark}

\begin{proposition} \label{propKC} Let $\bv\in \Upsilon_t$. If system \eqref{bc2uw}
\begin{equation}
\mb \Psi_{\bv} \mb u(\bv)   = 0
  \label{bc2uw'}
  \end{equation}
contains a Kirchhoff's condition
\begin{equation}
  \sum\limits_{j \in J_{\bv}} (\psi^j_{\bv,r} u^j_1(\bv) + \psi^j_{\bv,r} u^j_2(\bv))=0,
  \label{bc2'}
  \end{equation}
  with $\psi^j_{\bv,r} \neq 0$ for all $j\in J_{\mb v}$ and some $r \in\{1,\ldots, k_{\bv}\}$, then \eqref{sfC} is satisfied.
  \end{proposition}
  \begin{proof}
  Condition \eqref{bc2'} ensures that each entry of the $r$-th row of both $\widehat{\mb{\Psi}_{\bv}^{out}}$ and $\widehat{\mb{\Psi}_{\bv}^{in}}$ is 1 and thus the product of each column of $\widehat{\mb{\Psi}_{\bv}^{out}}$ with each column of $\widehat{\mb{\Psi}_{\bv}^{in}}$ is non-zero, which yields \eqref{sfC}.
  \end{proof}
    \begin{example} \label{ex38} Consider the model of \cite{Nic}, analysed in the framework of our approach in \cite[Example 5.12]{JBAB1},
 \begin{equation}
  \p_tp^{j}_1 + K^j \p_x p^j_2  =0, \quad  \p_tp^{j}_2 + L^j \p_x p^j_1  =0,
      \label{sysN}
   \end{equation}
   for $t>0, 0<x<1, 0\leq j\leq m,$ where $K^j>0,L^j>0$  for all $j$.
   For a given vertex $\bv,$ we define $(\mb p_1(\mb v),\mb p_2(\mb v)) = ((p^j_1(\mb v), p_2^j(\mb v))_{j\in J_{\mb v}},$ $\nu^j(\bv) =  -1$ if $ l_j(\bv) = 0$ and $\nu^j(\bv) =  1$ if $l_j(\bv)= 1,$ and  $T_{\bv} \mb p_2 (\bv)= (\nu^j(\bv)p^j_2(\mb v))_{j\in J_{\bv}}.$
      In this case $\alpha_j=1$ for any $j$ and thus for any vertex $\bv$ we need $|J_{\bv}|$ boundary conditions. We focus on $\bv$  with $|E_{\bv}|>1$. Then we split $\mbb R^{|J_{\bv}|}$ into $X_{\bv}$ of dimension $n_{\bv}$ and its orthogonal complement $X_{\bv}^\perp$ of dimension $l_{\bv} = |J_{\bv}|-n_{\bv}$ and require that
      $$
      \mb p_1(\bv) \in X_{\bv}, \quad T_{\mb v}\mb p_2(\bv) \in X^\perp_{\bv},
      $$
that is, denoting $I_1=\{1,\ldots,n_{\bv}\}$ and $I_2=\{n_{\bv}+1,\ldots,|J_{\bv}|\}$,
\begin{equation}\label{Nbc}
\sum\limits_{j\in J_{\bv}}\phi^j_r p^j_1(\bv) = 0, \quad r \in I_2,\quad
\sum\limits_{j\in J_{\bv}}\vhi^j_r \nu^j(\bv)p^j_2(\bv) = 0, \qquad r \in I_1,
\end{equation}
where $((\vhi^j_r)_{j\in J_{\bv}})_{r\in I_1}$ is a base in $X_{\bv}$ and  $((\phi^j_r)_{j\in J_{\bv}})_{r\in I_2}$ is a base in $X^\perp_{\bv}.$
 It is clear that, in  general, boundary conditions \eqref{Nbc} do not satisfy \eqref{sfC}. Consider $\bv$ such that each $\mb e^j$ incident to $\bv$ is parameterised so as $l_j(\bv) =0$ so that each $u^j_1(0)$ is outgoing and each $u^j_2(0)$ is incoming.  If we take $\phi_r^j = \vhi^j_r = \delta_{rj}$ and $L^j=K^j =1$ for $j\in J_{\bv}$, we obtain
\begin{align*}
p^r_1(0) &= u^r_1(0)+ u^r_2(0) =0, \quad r=n_{\bv}+1,\ldots,|J_{\bv}|,\\
p^r_2(0) &= u^r_1(0)- u^r_2(0)=0, \quad r=1,\ldots,n_{\bv}.
\end{align*}
Thus $\widehat{\mb \Psi_{\bv}^{out}}$ and $\widehat{\mb \Psi_{\bv}^{in}}$ are both the identity matrices and \eqref{sfC} is not satisfied.

On the other hand, the Kirchhoff condition,
\begin{equation}
\sum\limits_{j\in J_{\bv}} \nu^{j}(\bv) p^{j}_2(\bv) =0,
\label{nic'}
\end{equation}
see  \cite[Eqn (4)]{Nic}, satisfies \eqref{sfC}, as we have
\begin{align*}
0&=  \sum\limits_{j\in J_{\bv}}\nu^j(\bv)p^j_2(\bv) =  \sum\limits_{j\in J_{\bv}}\nu^j(\bv) (f^j_{+,2}(\bv)u^j_1(\bv) +f^j_{-,2}(\bv)u^j_2(\bv)) \\&= \sum\limits_{j\in J_{\bv}}\nu^j(\bv)\sqrt{K^jL^j} (u^j_1(\bv) -u^j_2(\bv)) \\
&= -\sum\limits_{j \in J^0_{\bv}} \sqrt{K^jL^j}u^j_1(0)
    -\sum\limits_{j\in J^1_{\bv}} \sqrt{K^jL^j}u^j_2(1) \\
    &\phantom{x}+\sum\limits_{j \in J^1_{\bv}} \sqrt{K^jL^j}u^j_1(1)
    +\sum\limits_{j\in J^0_{\bv}} \sqrt{K^jL^j}u^j_2(0),
\end{align*}
where we used \cite[Eqn 5.2]{JBAB1}. Thus the assumption of Proposition \ref{propKC} is satisfied.
\end{example}
\begin{example}
Let us consider the linearized Saint-Venant system,
\begin{equation}
\p_tp^j_1 = -V^j \p_x p^j_1 - H^j\p_x p^j_2, \quad \p_t p^j_2= -g \p_x p^j_1 -V^j \p_x p^j_2,
\label{SV}
\end{equation}
see \cite[Example 1.2]{JBAB1}, assuming that on each edge we have $\la^j_{\pm} = V^j \pm \sqrt{gH^j} >0$. Then we have
\begin{equation}
 \binom{p^j_1}{p^j_2} = \binom{f^j_{+,1} u^j_1 + f^j_{-,1}u^j_2}{f^j_{+,2}u^j_1 + f^j_{-,2}u^j_2} = \binom{H^j u^j_1+ H^j u^j_2}{\sqrt{gH^j}u^j_1 -\sqrt{gH^j}u^j_2}.\label{SNpguw}
 \end{equation}
We use the flow structure of \cite[Example 5.1]{KMN}, shown in Fig. \ref{fig1}, and focus on $\bv_1,$  where we need $2N-2$ boundary conditions which were given as
\begin{figure}
\centering
\begin{tikzpicture}[scale=0.85]
\draw [->] (-3,0) -- (-0.1,0);
\draw  [dashed, ->] (1.6,0.6) to [out=-45,in=45] (1.6,-0.6);
\draw (-0.1,0.3) node {\footnotesize{$\bv_1$}};
\draw(0,0)node{$\bullet$};
\draw (-3,0)node{$\bullet$};
\draw  [->](0,0)--(1.95,0.95);
\draw  [->](0,0)--(2,-1);
\draw[->](0,0)--(0.95,1.95);
\draw [->](0,0)--(0.95,-1.95);
 \draw (-3,-0.3) node {\footnotesize{$0$}};
 \draw (-3,0.3) node {$\bv_0$};
 \draw (1.4,0.4) node {\footnotesize{$e_3$}};
 \draw (1.3,-0.9) node {\footnotesize{$e_{N-1}$}};
 \draw (-1.5,-0.3) node {\footnotesize{$e_1$}};
 \draw (-0.3,-0.2) node {\footnotesize{1}};
  \draw (0.3,0.3) node {\footnotesize{$0$}};
  \draw (0.3,0) node {\footnotesize{$0$}};
  \draw (0.25,-0.3) node {\footnotesize{$0$}};
  \draw (-0.05,-0.25) node {\footnotesize{$0$}};
 \draw (1,2.2) node {\footnotesize{$\bv_2$}};
 \draw (1,1.8) node {\footnotesize{$1$}};
 \draw(1,2)node{$\bullet$};
 \draw(2,1)node{$\bullet$};
 \draw(2.05,-1.05)node{$\bullet$};
 \draw(1,2)node{$\bullet$};
 \draw (1.2,-1.8) node {\footnotesize{$\bv_{N}$}};
 \draw (2.2,-0.8) node {\footnotesize{$\bv_{N-1}$}};
 \draw (2.2,1.1) node {\footnotesize{$\bv_3$}};
 \draw(1,-2)node{$\bullet$};
 \draw (0.8,1.2) node {\footnotesize{$e_2$}};
 \draw (0.2,-1) node {\footnotesize{$e_{N}$}};
 \draw (1.9,-1.1) node {\footnotesize{$1$}};
 \draw (2,0.8) node {\footnotesize{$1$}};
 \draw (0.8,-2) node {\footnotesize{$1$}};
  \end{tikzpicture}
\caption{Starlike network of channels}\label{fig1}
\end{figure}
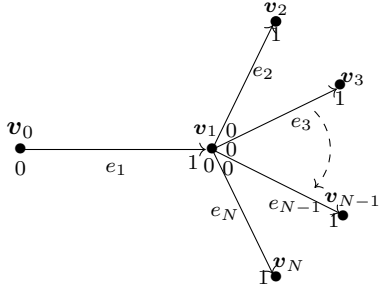
$$
p^j_1(0) = p^1_1(1), \quad p^j_2(0) = p_2^1(1),\qquad  j=2,\ldots, N.
$$
In terms of the Riemann invariants, they can be written as
\begin{align*}
&\left(\begin{array}{ccccccc}H^2&H^2&0&0&\ldots&0&0\\\sqrt{gH^2}&-\sqrt{gH^2}&0&0&\ldots&0&0\\\vdots&\vdots&\vdots&\vdots&\ddots&\vdots&\vdots\\
0&0&0&0&\ldots&H^N&H^N\\0&0&0&0&\ldots&\sqrt{gH^N}&-\sqrt{gH^N}\end{array}\right)\left(\begin{array}{c}u_1^2(0)\\u_2^2(0)\\\vdots \\ u_1^N(0)\\u_2^N(0)\end{array}\right)\\
&\phantom{xxxx}=\left(\begin{array}{cc}H^1&H^1\\\sqrt{gH^1}&-\sqrt{gH^1}\\\vdots\\
H^1&H^1\\\sqrt{gH^1}&-\sqrt{gH^1}\end{array}\right)\left(\begin{array}{c}u_1^1(1)\\u_2^1(1)\end{array}\right)
\end{align*}
and it is clear that \eqref{sfC} is satisfied.
\end{example}
\subsection{Graph reconstruction}
For a matrix $A = (a_{ij})_{1\leq i\leq n,1\leq j\leq m},$ denote by $\mb a^c_j, 1\leq j\leq m,$ the columns of $A$ and by $\mb a^r_i, 1\leq i\leq n,$ the set of its rows. Then we often write
\begin{equation}\label{cr}
A =   (\mb a^c_{j})_{1\leq j\leq m} = (\mb a^r_{i})_{1\leq i\leq n},
\end{equation}
that is, we represent the matrix as a row vector of its columns or a column vector of its rows. In particular, we write
\begin{align*}\mb \Xi_{out} &= (\xi^{out}_{ij})_{1\leq i\leq 2m, 1\leq j\leq 2m} = (\mb\xi^{out,c}_{j})_{1\leq j\leq 2m} = (\mb\xi^{out,r}_{i})_{1\leq i\leq 2m},\\\mb \Xi_{in} &= (\xi^{in}_{ij})_{1\leq i\leq 2m, 1\leq j\leq 2m}= (\mb\xi^{in,c}_{j})_{1\leq j\leq 2m} = (\mb\xi^{in,r}_{i})_{1\leq i\leq 2m}.\end{align*}
For any vector $\mb \mu = (\mu_1,\ldots,\mu_k)$, we define $\text{supp}\;\mb\mu =\{j\in \{1,\ldots,k\};\; \mu_j\neq 0\}$.
\begin{definition} We say that \eqref{syst1} is graph realizable if there is a graph $\Gamma =\{\{\bv_i\}_{1\leq i\leq r}, \{\mb e_k\}_{1\leq k\leq m}\}$ and a grouping of the column indices of $\mb\Xi$ into pairs $(j'_k,j_k^{''})_{1\leq k\leq m}$  such that \eqref{syst1a}  describes a flow along the edges $\mb e_k$ of $\Gamma$, which satisfies a general Kirchhoff's condition at each vertex of $\Gamma$.  In other words, \eqref{syst1} is graph realizable if there is a graph $\Gamma$ and a Kirchhoff's matrix $\mb\Psi$ such that \eqref{syst1a}, \eqref{syst1c} can be written, after possibly permutating rows and columns of $\mb\Xi,$ as \eqref{sysdiag}, \eqref{bcsplit2}.
\end{definition}
Before we formulate the main theorem, we need to introduce some notation. Let us recall that we consider the boundary system \eqref{syst1b1}
$$
\mb \Xi_{out}((\ups_j(0,t))_{j\in J^+},(\w_j(1,t))_{j\in J^-}) =- \mb \Xi_{in}((\ups_j(1,t))_{j\in J^+},(\w_j(0,t))_{j\in J^-}).
$$
 Let us emphasize that in this notation, the column indices on the left  and right hand side correspond to the values of the same function. To shorten notation, let us renumber them as $1\leq j\leq 2m$. As noted in Introduction, appropriate pairs of the columns would determine the edges of the graph $\Gamma$ that we try to reconstruct, hence the first step is to identify the possible vertices of $\Gamma$. For this, we first will try to construct a multi digraph $\mb \Gamma$ on which \eqref{syst1b1} can be written in the form \eqref{bcsplit2} for $(\mb\ups, \mb\w)$. Roughly speaking, this corresponds to $\mb\Xi_{out}$ and $\mb\Xi_{in}$ being composed (up to permutations) of non-communicating blocks corresponding to the vertices. Here, each $j$ should correspond to an arc $\me^j$ and the column $j$ on the left hand side corresponds to the outflow along $\me^j$ from a unique vertex, while the column $j$ on the right hand side corresponds to the inflow along $\me^j$ to a unique vertex. The vertices of $\mb\Gamma$ should be then determined by a suitable partition of the rows of $\mb \Xi_{out}$ (and of $\mb \Xi_{in}$).

 In the second step we will determine additional assumptions that allow $\mb\Gamma$ to be collapsed into a graph $\Gamma$ on which \eqref{syst1b1} can be written in the form \eqref{bcsplit2}.

Since we do not want \eqref{syst1b1} to be under- or over-determined, we assume
\begin{equation}
\forall_{1\leq j\leq 2m}\; \mb\xi^{out,c}_j\neq 0\;\text{and}\; \mb \xi^{out,r}_j \neq 0.
\label{assout}
\end{equation}
Our strategy is to treat $\mb\Xi_{out}$ and $\mb \Xi_{in}$ as the outgoing and incoming incidence matrices of a multi digraph with vertices `smeared' over subnetworks of flow connections. Thus we assume that
\begin{equation}
\msf A\!\! :=\!\!\widehat {\left(\widehat{\mb \Xi_{out}}\right)^T\widehat{\mb \Xi_{in}}}\;\;\text{is\;the\;adjacency\;matrix\;of\;a\;line\;digraph}.
\label{MAss}
\end{equation}
For $\msf A$, let $V^{out}_j$ and $V^{in}_i$ be  groups of row and column indices, respectively, potentially outgoing from (respectively incoming to) a vertex, see Appendix \ref{appA}. We introduce
$$
I :=\{i\in \{1,\ldots,2m\};\; \mb\xi^{in,r}_i=0\}
$$
and assume that
\begin{equation}\forall_{i \in I}\exists_{1\leq j\leq M'}\;\spp \mb \xi^{out,r}_i \subset V^{out}_j,\label{ass3}
\end{equation}
where $\spp$ denotes the support of the vector. In the next proposition we shall show that $V^{out}_j$ and $V^{in}_i$ determine a partition of the row indices into sets that can be used to define vertices. The idea is that if supports of columns of $\mb \Xi_{out}$ or $\mb \Xi_{in}$ overlap, the arcs determined by these columns must be incident to the same vertex.
\begin{proposition}
If \eqref{assout}, \eqref{MAss} and \eqref{ass3} are satisfied, then the sets
\begin{equation}\{\{\mc V_i\}_{1\leq i\leq n}, \mc V_S\},\label{rowpart1}
\end{equation}
 where
\begin{align}
\mc V_S &=\{i \in \{1,\ldots,2m\};\; \mathrm{supp\;} \mb \xi^{out,r}_i \subset V^{out}_{M'}\}, \label{parts}\\
\mc V_i &= \bigcup\limits_{s \in V^{out}_i}\mathrm{supp\;} \mb {\xi}^{out,c}_s,\;1\leq i\leq n, \label{partint}
\end{align}
form a partition of the row indices of both $\mb \Xi_{out}$ and $\mb \Xi_{in}$ such that if for any $j$, $\mathrm{supp\;}\mb\xi^{out,c}_j\cap \mc V_i\neq \emptyset$ for some $i=1,\ldots,n, S$, then $\mathrm{supp\;}\mb\xi^{out,c}_j\subset \mc V_i,$ and if for any $k$, $\mathrm{supp\;}\mb\xi^{in,c}_k\cap \mc V_l\neq \emptyset$ for some $l= 1,\ldots,n, S$, then $\mathrm{supp\;}\mb\xi^{in,c}_k \subset \mc V_l.$\label{admitg}
\end{proposition}
\begin{proof}
By \eqref{MAss}, $\msf A$ is the adjacency matrix of $L(\mb \Gamma)$  for some multi digraph $\mb\Gamma$. As explained in Appendix \ref{appA}, we can reconstruct $\mb\Gamma$ with the transient vertices defined in a unique way and admissible sources and sinks. Let us fix such a construction. Then we have the sets $\{V^{in}_j\}_{1\leq j\leq n}$ and $\{V^{out}_i\}_{1\leq i\leq n}$ of incoming and outgoing arcs determining any transient vertex. Further, we have (possibly) the sets $V^{in}_{N'}$ and $V^{out}_{M'}$ that group the arcs incoming to sink(s) and, respectively, outgoing from  source(s). Since $\msf A$ represents all arcs, the same decomposition is valid for $\mb \Xi_{out}$ and $\mb \Xi_{in}$, that is, we have  subdivisions $\{V^{out}_i\}_{1\leq i\leq M'}$ and $\{V^{in}_j\}_{1\leq j\leq N'}$ of the columns  of $\mb \Xi_{out}$ and $\mb \Xi_{in}$, respectively, and hence the correspondence of the columns  with the vertices. Thus we have to show that \eqref{rowpart1} is a partition  of the rows of $\mb\Xi_{out}$ and $\mb \Xi_{in}$ satisfying the conditions of the proposition.

Let us recall that the entry $\msf a_{ij}$ of $\msf A$ is defined by $\widehat{\widehat{\mb {\xi}^{out,c}_i}\cdot \widehat{\mb {\xi}^{in,c}_j}}$ and if a row $k$ of $\msf A$ is zero, that is, it represents a source, then there is a zero row in  ${\mb \Xi_{in}}$. Indeed, since, by \eqref{assout}, $\spp \mb {\xi}^{out,c}_k \neq \emptyset$, there is a nonzero entry, say, $\xi^{out}_{lk}$ and thus we must have $\xi^{in}_{lj} =0$ for any $j$. So, to every zero row in $\msf A,$ there corresponds a zero row in $\mb \Xi_{in}$. There may be, however, other zero rows in $\mb\Xi_{in}$.

To determine the rows in $\mb\Xi_{out}$, corresponding to sources, we consider \eqref{ass3}.
First we note that any $j$ of that assumption, if it exists, is determined in a unique way as the sets $V^{out}_j$ are not overlapping. Next we observe that if $\spp \mb \xi^{out,r}_i \subset V^{out}_j$ and $\spp \mb \xi^{out,r}_i\cap \spp \mb \xi^{out,r}_k \neq \emptyset$, then $\spp \mb \xi^{out,r}_k \subset V^{out}_j$. Indeed, if $\spp \mb \xi^{out,r}_k \subset V^{out}_p$ and $l\in \spp \mb \xi^{out,r}_i\cap \spp \mb \xi^{out,r}_k$, then $l \in V^{out}_j\cap V^{out}_p$ which implies $V^{out}_j = V^{out}_p.$ Thus we can define the set

$$\mc V_S =\{i \in \{1,\ldots,2m\};\; \spp \mb \xi^{out,r}_i \subset V^{out}_{M'}\}.$$

For any $i\in \mc V_S$ and $q\in \spp \mb \xi^{out,r}_i,$ we have $\spp \mb\xi^{out,c}_q\subset I,$ as otherwise there would be a nonzero product $\mb\xi^{out,c}_q\cdot \mb\xi^{in,c}_s$ for some $s$ as $\mb\xi^{in,r}_{t} \neq 0$ for each $t\notin I$. Then, let there be $k$ such that $\spp \mb\xi^{out,c}_p \cap \spp \mb\xi^{out,c}_j\ni k$ for some $j \in V^{out}_{M'}$. This means, by \eqref{ass3}, that $\spp\mb \xi^{out,r}_k \subset V^{out}_{M'}$ and hence $p \in  V^{out}_{M'}$. Consider any nonzero element of $\mb \xi^{out,c}_p,$ say, $\xi^{out}_{lp}\neq 0$. By the above argument, $l \in I$. If $\spp \mb\xi^{out,r}_l \subset V^{out}_{M'}$, then $l\in \mc V_S$. If not, $\spp \mb\xi^{out,r}_l\cap V^{out}_j\neq \emptyset$ for some $j\neq M'$ which contradicts \eqref{ass3}. Thus, $\mc V_S$ satisfies the first part of the statement. The second part is void as there is no $\mb\xi^{in,c}_j$ with $\spp\mb\xi^{in,c}_j\cap \mc V_S\neq \emptyset.$
Therefore all indices $i\in \mc V_S,$ that is, such that  $\spp \mb \xi^{out,r}_i \subset V^{out}_{M'},$ determine a source as there is no connection to any inflow.

Now, consider the indices $i\in I\setminus \mc V_S$. Then, again by \eqref{ass3}, for any $i\in I\setminus \mc V_S$ there is a unique $j\neq M'$ such that $\spp \mb \xi^{out,r}_i \subset V^{out}_{j}$, that is, such an $i$ belongs to the vertex determined by  $V^{out}_{j}$. This determines a partition of $I$ corresponding to the vertices (recall that there are no zero rows in $\mb\Xi_{out}$ and so each row must belong to a vertex).

 Next we associate the remaining rows in  $\mb \Xi_{out}$ and $\mb \Xi_{in}$ with the vertices. Consider $V^{out}_i$ and $V^{in}_j$ for some $1\leq i\leq n$ and $j$ defined by \eqref{vertid}. The non-zero entries $\msf a_{pq}$ of $\msf A$, where $p\in V^{out}_i$ and $q\in V^{in}_j$, occur whenever $\spp \mb {\xi}^{out,c}_p \cap \spp \mb {\xi}^{in,c}_q\neq \emptyset$. Hence, the rows with indices $k \in \spp \mb {\xi}^{out,c}_p \cap \spp \mb {\xi}^{in,c}_q$ must belong to a vertex through which the incoming arc $\me^q$ communicates with the outgoing arc $\me^p$. Since all nonzero entries in, respectively, $\spp \mb {\xi}^{out,c}_p$ and $ \spp \mb {\xi}^{in,c}_q$ reflect non-zero outflow along $\me^p,$ respectively, inflow along $\me^q$, $\spp \mb {\xi}^{out,c}_p$ and $ \spp \mb {\xi}^{in,c}_q$ must belong to the same vertex. Since the same is true for any indices from $V^{out}_i$ and $V^{in}_j$, plausible partitions of row indices of $\mb\Xi_{out}$ and $\mb\Xi_{in}$ defining vertices are,
$$
\mc V^{out}_i = \bigcup\limits_{s \in V^{out}_i}\spp \mb {\xi}^{out,c}_s, \qquad \mc V^{in}_j = \bigcup\limits_{q \in V^{in}_j}\spp \mb {\xi}^{in,c}_q.
$$
We first observe that if $V^{out}_i$ and $V^{in}_j$ determine the same transient vertex, then
\begin{equation}
\mc V^{out}_i\setminus \{s \in \mc V^{out}_i;\; \mb \xi^{in,r}_s =0\} = \mc V^{in}_j.
\label{vinvout}
\end{equation}
Indeed, let $p \in \mc V^{out}_i\setminus \{s \in \mc V^{out}_i;\; \mb \xi^{in,r}_s =0\}$. Then there is $s \in V^{out}_i$ such that $p \in \spp \mb {\xi}^{out,c}_s$. Since $p\notin \{s \in \mc V^{out}_i;\; \mb \xi^{in,r}_s =0\}$, there is $q$ such that $ \xi^{in}_{pq}\neq 0$ and thus $\widehat{\mb \xi^{out,c}_s} \cdot \widehat{\mb \xi^{in,c}_q} \neq 0$. Hence, $q \in V^{in}_j$ and consequently $p \in \mc V^{in}_j$. The converse can be proved in the same way by using assumption \eqref{assout} since if $p\in \mc V^{in}_j$ then, by construction, $p$ must belong to a support of some  $\mb {\xi}^{in,c}_q$ and thus cannot be in $\{s \in \mc V^{out}_i;\; \mb \xi^{in,r}_s =0\}$. As we see, if $\mc V^{out}_i$ contains rows $\mb\xi^{out,r}_k$ with $k\in I\setminus \mc V_S$, then these rows satisfy $\spp \mb\xi^{out,r}_k \subset V^{out}_i$. If we add the indices of such rows to $\mc V^{in}_j$ with $V^{in}_j$ determining the same vertex as $V^{out}_i$, then such an augmented $\mc V^{in}_j$ will be equal to $\mc V^{out}_i$ and thus we use can use \eqref{rowpart1} to denote the partition of $\{1,\ldots,2m\}$ into $\mc V_1^{out}, \ldots, \mc V^{out}_n,\mc V_S$.

We easily check that this partition satisfies the conditions of the proposition. We have already checked this for $\mc V_S$. So, let $\spp\mb \xi^{out,c}_q \cap \mc V^{out}_i\neq\emptyset $ for some $1\leq i\leq n$, then there is $s\in V^{out}_i$ such that $k\in \spp\mb \xi^{out,c}_q \cap \spp\mb \xi^{out,c}_s$. Clearly, $k\notin \mc V_S$ by the construction of $\mc V^{out}_i$. If $k \in I\setminus \mc V_S$, then $q \in V^{out}_i$ by assumption \eqref{ass3} and hence $\spp\mb\xi^{out,c}_q \subset \mc V^{out}_i$. If $k\notin I$, then $\widehat{
\mb \xi^{out,c}_s}\cdot \widehat{\mb \xi^{in,c}_p}\neq 0$ for some $p$ but then also $\widehat{\mb\xi^{in,c}_p}\cdot \widehat{\mb\xi^{out,c}_q}\neq 0$ and hence $p \in V^{in}_j$, yielding $q \in V^{out}_i$ and consequently $\spp \mb\xi^{out,c}_q \subset \mc V^{out}_i.$ Similarly, if $\spp\mb\xi^{in,c}_p \cap \mc V^{out}_i\neq\emptyset$, then there is $k \in \spp\mb\xi^{in,c}_p\cap \spp\mb\xi^{out,c}_q$ for some $q \in V^{out}_i$. But then, immediately from the definition, $\spp\mb\xi^{in,c}_p\subset \mc V^{in}_j\subset \mc V^{out}_i$ by \eqref{vinvout}.
\end{proof}
We note that \eqref{rowpart1} does not contain  rows corresponding to sinks  and they must be added following the rules described in Appendix \ref{appA}. With such an extension,  we consider the multi digraph $\mb \Gamma$, determined by
\begin{equation}
\{\{\mc V_i\}_{1\leq i\leq n}, \mc V_S, \mc V_Z\},\quad \{\{V^{out}_i\}_{1\leq i\leq n}, V^{out}_{M'}, \emptyset\}, \quad \{\{V^{in}_{j_i}\}_{1\leq i\leq n},\emptyset, V^{in}_{N'}\},
\label{mbgamma}
\end{equation}
where the association $i\mapsto j_i$ is defined in \eqref{vertid}. By construction, if we take the triple $\mc V_i,V^{out}_i,V^{in}_{j_i}$, $1\leq i\leq n$,  it determines a transient vertex, the outgoing arcs given by the indices of columns in  $\mb \Xi_{out}$ and the incoming arcs given by the indices of columns in $\mb\Xi_{in}$. Similarly, the pair $\mc V_S, V^{out}_{M'}$ determines the sources and all outgoing arcs, while the set of incoming arcs is empty. Thus, if we denote by $\mb\Xi^i_{out}$ and $\mb\Xi^i_{in}$ the submatrices of $\mb\Xi_{out}$ and $\mb\Xi_{in}$ consisting of the rows with indices in $\mc V_i$ and columns in $V^{out}_i$ and $V^{in}_{j_i}$, respectively, with an obvious modification for $\mc V_S$, then \eqref{syst1b1} decouples into $n$ (or  $n+1$) independent systems
\begin{equation}\label{Gamsys}
\begin{split}
&\mb \Xi^i_{out}((\ups_j(0,t))_{j\in J^+\cap V^{out}_i},(\w_j(1,t))_{j\in J^-\cap V^{out}_i})\\& \phantom{xxx}=- \mb \Xi^i_{in}((\ups_j(1,t))_{j\in J^+\cap V^{in}_{j_i}},(\w_j(0,t))_{j\in J^-\cap V^{in}_{j_i}}), \quad 1\leq i\leq n,\\
&\mb \Xi^S_{out}((\ups_j(0,t))_{j\in J^+\cap V^{out}_{M'}},(\w_j(1,t))_{j\in J^-\cap V^{out}_{M'}}) = 0.
\end{split}
\end{equation}
 This system can be seen as a Kirchhoff system on the multi digraph $\mb \Gamma$ but we need to collapse $\mb \Gamma$ to a graph $\Gamma$ on which \eqref{Gamsys} can be written as \eqref{bcsplit2}. We observe that the question naturally splits into two problems -- one is about collapsing the graph, while the other is about grouping the components of $(\mb\ups,\mb\w)$ into pairs compatible with the parametrization of $\Gamma$.

 Let $\msf A$ be the adjacency matrix of $L(\mb \Gamma)$, with \eqref{ass3} and \eqref{assout} satisfied. As in Appendix \ref{appA}, we can construct outgoing and incoming incidence matrices $A^+$ and $A^-$ but these are uniquely determined only if there are no sources and sinks. We have, however, an additional piece of information about sources.

 If we grouped all sources into one node, as before Proposition \ref{1source}, then, by  Lemma \ref{lemmbbc2}, the flow connectivity in this source was given by
 $$
 \msf C_{\bv}:= \widehat{\left(\widehat{\mb \Xi_{out}^{S}}\right)^T\widehat{\mb \Xi_{out}^{S}}}.
 $$
  However, such a matrix would not necessarily satisfy assumption \eqref{assirr}. Thus,  we separate the arcs into non-communicating groups, each determining a source satisfying \eqref{assirr}.   For this, by simultaneous permutations of rows and columns, $\msf C_{\bv}$ can be written as
  \begin{equation}\label{xis}
  \mb\Xi^S = diag\{ \mb \Xi^S_{i}\}_{1\leq i\leq k},
  \end{equation}
  where $k$ may equal 1. Since the simultaneous permutation of  rows and columns  is given as $P\msf C_{\bv}P^T,$ where $P$ is a suitable permutation matrix, \cite[p. 140]{Mey},  we see that $\mb\Xi^S$ is a symmetric matrix, along with $\msf C_{\bv}$.
By \cite[Sections III $\S$ 1 and  III $\S$ 4]{Gant}, $\mb\Xi^S$ is irreducible if and only if it cannot be transformed by simultaneous row and column permutations to the form \eqref{xis} with $k>1$ (since $\mb\Xi^S$ is symmetric, all off-diagonal blocks must be zero). Then $\mb \Xi^S$ can be reduced to the canonical form, \cite[Section III, Eq. (68)]{Gant}, in which each $\mb\Xi^S_i$ is irreducible. If \eqref{xis} is in the canonical form, then we say that $\mb \Xi^S$ allows for $k$ sources, each satisfying \eqref{assirr}. The indices of the columns contributing to the blocks define the $k$ non-communicating sources $\mc V_{S_1}, \ldots, \mc V_{S_k}$ in $\mb \Gamma$, which we denote $V^{out}_{S_1}, \ldots, V^{out}_{S_k}$, $V^{out}_{M'} = V^{out}_{S_1}\cup \ldots\cup V^{out}_{S_k}$. Finally, define $\mb \xi^{out,r} = (\xi^{out,r}_{S_i j})_{1\leq i\leq k,1\leq j\leq 2m}$ by
  $$
  \xi^{out,r}_{S_i j} = \left\{\begin{array}{lcl} 1 &\text{if}& j\in V^{out}_{S_i},\\0&\text{otherwise}.&\end{array}\right.
  $$
  For the sinks, it is simpler as there is no constraining information  from \eqref{syst1b1}. We have columns with indices in $V^{in}_{N'}$ corresponding to sinks. These are zero columns in $\mb\Xi_{in}$ but the columns with these indices in $\mb\Xi_{out}$ have nonempty supports and thus we can determine from which vertices they are outgoing. Let us denote
  \begin{equation}
  V^{in}_{\mc V_i} = \{j\in V^{in}_{N'};\; \spp \mb\xi^{out,c}_j\cap \mc V_i\neq \emptyset\}, \quad i=1,\ldots,n, S_1,\ldots, S_k.\label{zrodla}
  \end{equation}
    For each $i$ we consider a partition
    \begin{equation}
    V^{in}_{\mc V_i} = V^{in}_{i,{1}}\cup\ldots\cup V^{in}_{i,{l_i}},\quad  i=1,\ldots,n, S_1,\ldots, S_k,
    \label{zrodla1}
    \end{equation}
    where $l_i\leq |V^{in}_{\mc V_i}|,$ into non-overlapping sets $V^{in}_{i,{l}}$, $1\leq l\leq l_i$. Then we define sinks  $\mc V_{i,{l}}$  as the heads of the arcs with indices from $V^{in}_{i, {l}}$; we have $n_z = l_1+\cdots + l_{S_k}$ sinks.  Then, as above, define $\mb \xi^{in,r} = (\xi^{in,r}_{\{i,{l}\}, q})_{i\in \{1,\ldots,S_k\},l\in \{1,\ldots,l_i\},q\in \{1,\ldots, 2m\}}$ by
  $$
  \xi^{in,r}_{\{i,{l}\}, q} = \left\{\begin{array}{lcl} 1 &\text{if}& q\in V^{in}_{i,{l}},\\0&\text{otherwise}.&\end{array}\right.
  $$
  \begin{remark} We expect  $|V^{in}_{\mc V_i}|, i=1,\ldots,n, S_1,\ldots, S_k,$ to be even numbers and  \eqref{zrodla1} to represent a partition of $V^{in}_{\mc V_i}$ into pairs so that   $l_i = |V^{in}_{\mc V_i}|/2.$
  \end{remark}
Then, as in Remark \ref{remA24}, the incoming and outgoing incidence matrices are
$$
A^+ = \left(\begin{array}{c} \msf A^+\\\mb 0\\\mb\xi^{in,r}\end{array}\right), \qquad
A^- = \left(\begin{array}{c} \msf A^-\\\mb\xi^{out,r}\\\mb 0\end{array}\right)
$$
which, by a suitable permutation of columns moving the sources and the sinks to the last positions, can be written, respectively, as
\begin{equation}
\left(\begin{tabular}{c|c|c|c} $\msf A^+_T$&$\msf A^+_S$&0&0\\
\hline
0&0&0&0\\
\hline
0&0&$\msf Z_S$&$\msf Z$\end{tabular}\right)\;\text{and}\; \left(\begin{tabular}{c|c|c|c} $\msf A^-_T$&$0$&$0$&$\msf A^-_Z$\\
\hline
0&$\msf S$&$\msf S_Z$&0\\
\hline
0&$0$&0&0\end{tabular}\right).
\label{ApAm}
\end{equation}
Both matrices have $2m$ columns and $\mc n :=n + k +  n_z$ rows. Hence, as shown in Remark \ref{remA24},
 the adjacency matrix of the full multi digraph $\mb \Gamma$ is given by
\begin{equation}
A(\mb\Gamma) =(a_{ij})_{1\leq i,j\leq \mc n}= A^+(A^-)^T = \left(\begin{tabular}{c|c|c} $\msf A^+_T(\msf A^-_T)^T$&$\msf A^+_S\msf S^T$&$0$\\
\hline
0&0&0\\
\hline
$\msf Z(\msf A_Z^-)^T$&$\msf Z_S(\msf S_Z)^T$&0\end{tabular}\right),\label{agama1}
 \end{equation}where the dimensions of the blocks in the first row are, respectively, $n\times n$, $n\times k$ and $n\times n_z$, in the second row, $k\times n$, $k\times k$ and $k\times n_z$ and in the last one, $n_z\times n$, $n_z\times k$ and $n_z\times n_z$. Thus, if $a_{ij}$ is in the block $(p,q), 1\leq p,q\leq 3, $ then $a_{ji}$ will be in the block $(q,p)$.

  Consider a nonzero pair  $(a_{ij},a_{ji})$ of entries of $A(\mb \Gamma)$. If, say, $a_{ij} = h$, then it means that the $i$-th row of $ A^+$ and $j$-th row of $A^-$ have entry 1 in the same $h$ columns, that is, there are exactly $h$ arcs coming from $\bv_j$ to $\bv_i$.  Similarly, if $a_{ji} = e$, then   there are exactly $e$ arcs coming from $\bv_i$ to $\bv_j$. Conversely, if there are $h$ arcs from $\bv_j$ to $\bv_i$ and $e$ arcs from $\bv_i$ to $\bv_j$, then $(a_{ij},a_{ji})=(h,e).$ In particular,  $h+e =2$ if and only if there are two arcs between $\bv_j$ and $\bv_i$ running either concurrently or countercurrently.  Since the columns of $A^+$ and $A^-$ are indexed in the same way as that of $\mb\Xi_{out}$ and  $\mb\Xi_{in}$, the pair  $(a_{ij},a_{ji})$ determines the rows $\mb a^{+,r}_i$ and $\mb a^{+,r}_j$ of $A^+$ and thus  the indices
  \begin{equation}
  \begin{split}
  &(a_{ij}\mapsto \{k^{ij}_1,\ldots, k^{ij}_h\}, a_{ji}\mapsto \{k^{ji}_1,\ldots, k^{ji}_e\}) \\
  &= (a_{ij}\mapsto \spp \mb a^{+,r}_i\cap \mb a^{-,r}_j, a_{ji}\mapsto\spp \mb a^{+,r}_j\cap \spp \mb a^{-,r}_i)\end{split}
  \label{edges}
  \end{equation}
   of columns of $\mb\Xi_{out}$ (and of $\mb\Xi_{in}$).

 \begin{theorem}
System \eqref{syst1b1} is graph realizable with Kirchhoff's conditions satisfying \eqref{ass1} and \eqref{sfC} for $\bv\in \Upsilon_t$ and \eqref{assirr} for $\bv\in\Upsilon_s$ if and only if, in addition to \eqref{assout}, \eqref{MAss} and \eqref{ass3},
 there is a partition \eqref{zrodla1} such that $A(\mb\Gamma)$ defined by \eqref{agama1} satisfies
 \begin{enumerate}
\item\label{assumption1} \, for any $1\leq i,j\leq \mc n$, $a_{ii}=0$ and  $(a_{ij},a_{ji})$ is in one of the following form
  \begin{equation}
  (2,0),\; (1,1),\; (0,2)\; \text{or}\; (0,0);\label{form}
  \end{equation}
  \item \,if  $(a_{ij},a_{ji})$ determines the indices $k$ and $l$ according to \eqref{edges}, then
  \begin{equation}
    \begin{tabular}{ll}
   if $(a_{ij},a_{ji}) = (2,0)$ or $(0,2)$, &then  $k,l \in J^+$ or $k,l \in J^-$\\& and $c_k\neq c_l,$\\
  if $(a_{ij},a_{ji}) = (1,1),$  &then  $k\in J^+$ and $l \in J^-$\\
  & or $k\in J^-$ and $l \in J^+$.
  \end{tabular}\label{edgeid}
  \end{equation}
\end{enumerate}

\end{theorem}

\begin{proof} \textit {Necessity.} Let us consider the Kirchhoff system \eqref{bcsplit2}. By construction, both matrices $\mb\Psi^{out}$ and $\mb\Psi^{in}$ are in block diagonal form with equal row dimensions of the blocks. We consider the problem already transformed to $\mb\Gamma$.  We note that each arc's index must appear twice in $\mb\Psi$ -- once in $\mb\Psi^{out}$ and once in $\mb\Psi^{in}$ (if there are sinks, the indices of incoming arcs will correspond to the zero columns). Further, whenever  column indices $k$ and $l$ appear in the blocks of, respectively,  $\mb\Psi^{out}$ and $\mb\Psi^{in}$, then $\me^l$ is incoming to, while $\me^k$ is outgoing from, the same vertex (and not any other). Thus, by \eqref{sfC}, the matrix
$$\tilde A = (\tilde a_{ij})_{1\leq i,j\leq 2m}= \widehat{(\widehat{\mb\Psi^{out}})^T\widehat{\mb\Psi^{in}}}$$
is block diagonal with blocks of the form $\msf C_{\mb v} = \mb 1_{\bv},$ except for zero rows corresponding to the sources and zero columns corresponding to sinks. In general, however, the column indices in $\mb\Psi^{out}$ and $\mb\Psi^{in}$ do not correspond to the indices of the arcs they represent. Precisely,  $\tilde a_{ij}= 1$ if and only if there is a vertex for which the arc $\me^{j'}$ is incoming and $\me^{i'}$ is outgoing, where $j'$ and $i'$ are the indices of the arcs that correspond to the columns $j$ and $i$ of, respectively  $\mb\Psi^{in}$ and $\mb\Psi^{out}$.  To address this, we construct $\mb\Xi_{out}$ and $\mb \Xi_{in}$  as
$$
\mb\Xi_{out} = \mb\Psi^{out}P, \qquad \mb \Xi_{in}=\mb\Psi^{in}Q,
$$
where $P$ and $Q$ are permutation matrices, so that in both matrices the column indices $1,\ldots,2m$ correspond to $\me^1,\ldots,\me^{2m}$. Hence
$$
A=(a_{ij})_{1\leq i,j\leq 2m} :=\widehat{(\widehat{\mb\Xi_{out}})^T\widehat{\mb\Xi_{in}}} = \widehat{(\widehat{\mb\Psi^{out}}P)^T\widehat{\mb\Psi^{in}}Q} = P^T\tilde AQ
$$
is a matrix where the indices $1,\ldots,2m$ of both the columns and the rows  correspond to $\me^1,\ldots,\me^{2m}$. Since $\Gamma$ did not have loops, it is clear that $a_{ii} =0$ for all $i=1,\ldots,2m.$ It is also clear that any two columns (or rows) of $\tilde A$ are either equal or orthogonal and this property is preserved by permutations of columns and of rows. Hence, by Proposition \ref{propGR}, $A$ is the adjacency matrix of a line digraph and hence \eqref{MAss} is satisfied. Since the arcs' connections given by $A$ and $\tilde A$ are the same, we see that $A$ is equal to the adjacency matrix of the line graph of $\mb \Gamma$. Therefore, the transient vertices determined by $A$ are the same as in $\mb \Gamma$ (and hence in $\Gamma$). On the other hand, as we know, $A$ does not determine the structure of sources and sinks in $\mb \Gamma$.   The fact that \eqref{assout} is satisfied is a consequence of \eqref{ass1}. For \eqref{ass3},  we recall, see Appendix \ref{appA}, that the sets $V_j^{out}$ group together the indices representing arcs $\me^k$ outgoing from a single vertex, thus they correspond to the blocks $\mb\Psi^{out}_v$ in the matrix $\mb\Psi^{out}$ and therefore \eqref{ass3} is satisfied, even for any $i$.
Next, since $\mb\Gamma$ has been constructed from $\Gamma$, the structure of the blocks in $\mb\Psi^{out}$ corresponding to sources ensures that, after permutations, their entries will coincide with $\mb\Xi^S_{out}$ and thus \eqref{xis} will hold with the blocks in \eqref{xis} exactly  corresponding to the sources in $\mb\Gamma,$ on account of \eqref{assirr}. Similarly, in $\Gamma$ the sinks are determined and thus we have groupings of pairs of the arcs (a partition of the set of indices corresponding to arcs incoming to sinks) coming from transient vertices or sources to sinks and thus the constructions  \eqref{ApAm} and \eqref{agama1} are completely determined. Then we observe that whenever we have a source $\bv$, then the arcs outgoing from $\bv$ must be coming in pairs with a single pair coming from $\mb v$ to any other possible vertex, meaning that the respective entry in $A(\mb\Gamma)$ must be either $(2,0)$ or $(0,2).$ Similar argument holds for the sinks.  Since the problem comes from a graph, by construction, the orientation of the flows is consistent with the parametrization.

\textit{Sufficiency}. Given \eqref{syst1b1}, we have flows $((\ups_j)_{j\in J^+}, (\w_j)_{j\in J^-})$ defined on $(0,1).$   Assumptions \eqref{assout}, \eqref{MAss} and \eqref{ass3} ensure, by Proposition \ref{admitg}, the existence of a multi digraph $\mb\Gamma$ on which \eqref{syst1b1} can be localized to decoupled systems at vertices and written as \eqref{Gamsys}. Precisely speaking, \eqref{MAss} associates the indices of incoming components of the solution at a vertex with incoming arcs and similarly for the indices of the outgoing components. Therefore, if an arc $\me^p$ runs from $\bv_j$ to $\bv_i$, then the flow occurs from $\bv_j$ to $\bv_i$, that is, if $p\in J^+$, then the flow on $\me^p$ is given by $\ups_p$ with $\ups_p(0)$ at $\bv_j$ and $\ups_p(1)$ at $\bv_i$ and analogous statement holds for $p\in J^-$. In other words, the index $p$ of the arc $\me^p$ running from $\bv_j$ to $\bv_i$ determines the orientation of the parametrization: $0\mapsto \bv_j$ and  $1\mapsto \bv_i$ if $p\in J^+$ and  $0\mapsto \bv_i$ and  $1\mapsto \bv_j$ if $p\in J^-.$

Now, assumption \ref{assumption1} ensures that there are no loops at vertices and that between any two vertices there are either two arcs or none. If $a_{ij}$ is an entry in $\msf A^+_S\msf S^T$, $\msf Z(\msf A_Z^-)^T$ or $\msf Z_S(\msf S_Z)^T$, then $a_{ij}=0$ or $a_{ij}=2$ and, , by the  dimensions of the blocks,  respectively, $a_{ji}\in\{0,2\}$ and $a_{ji}=0$. On the other hand, if $a_{ij}$ is an entry in $\msf A^+_T(\msf A^-_T)^T$, then it can take any value $0,1$ or $2$ and the $a_{ji}$ equals, respectively, $2$ or $0$, $1,\,0$. Thus, double arcs indexed, say, by $(k,l)$, between vertices could be combined into edges of an undirected graph (with no loops and multiple edges). However, in this way we construct a combinatorial graph which does not take into account that if $\me^k$ and $\me^l$ are combined into one edge $\mb e$ of $\Gamma$, their orientations must be the same.  Thus, if $(a_{ij},a_{ji}) = (2,0)$ determines the pair of indices $k,l$ according to \eqref{edgeid}, then both  arcs $\me^k$ and $\me^l$ of $\mb \Gamma$ run from $\bv_j$ to $\bv_i$  and the components $k$ and $l$ of the solution  flow concurrently along $\mb e$. By assumption,  $k,l \in J^+$ or  $k,l \in J^-$. In the first case, we  associate $\bv_j$ with 0 and $\bv_i$ with 1 and we have $(\ups_{k}, \ups_{l})$ on $\mb e$, in agreement with the orientation. Otherwise, we  associate $\bv_j$ with 1 and $\bv_i$ with 0 and we have $(\w_{k}, \w_{l})$ on $\mb e$. On the other hand, if  $(a_{ij},a_{ji})=(1,1)$ then, by assumption, either $k\in J^+$ and $l\in J^-$ or $k\in J^-$ and $l\in J^+$ and the components $k$ and $l$ flow countercurrently.  Again, in the first case, $k\in J^+$ and $\me^k$ running from $\bv_j$ to $\bv_i$ requires $\bv_j$ to be associated with  0 and $\bv_i$ with 1, while
$l\in J^-$ and $\me^l$ running from $\bv_i$ to $\bv_j$  also requires $\bv_j$ to be associated with  0 and $\bv_i$ with 1. Thus, we have $(\ups_{k}, \w_{l})$ on $\mb e$. Otherwise, we  associate $\bv_j$ with 1 and $\bv_i$ with 0 and we have $(\w_{k}, \ups_{l})$ on $\mb e$.

Finally, the assumption $c_k\neq c_l$ in the first case of assumption \eqref{edgeid} ensures that the resulting system is hyperbolic on each edge.
\end{proof}

\begin{example}
Let us consider the system
\begin{equation}\label{s1}
\begin{split}
\p_t\ups_{j}+c_j\p_x\ups_j &=0,\quad 1\leq j\leq 4,\\
\p_t\w_{j}-c_j\p_x\w_j &=0,\quad 5\leq j\leq 6,
\end{split}
\end{equation}
where $c_j>0,$ with boundary conditions
\begin{equation}
\left(\!\!\begin{array}{cccccc}0&1&1&0&0&0\\
1&0&0&0&0&0\\
1&1&0&0&0&0\\
0&0&1&1&0&0\\
0&0&0&0&1&0\\
0&0&0&0&0&1\end{array}\!\!
\right)\left(\!\!\!\begin{array}{c}\ups_1(0)\\\ups_2(0)\\\ups_3(0)\\\ups_4(0)\\\w_5(1)\\\w_6(1)\end{array}\!\!\!\right) = \left(\!\!\begin{array}{cccccc}0&0&0&0&0&0\\
0&0&0&0&0&0\\
0&0&0&0&0&0\\
0&0&0&0&0&0\\
1&1&0&0&0&1\\
0&0&1&1&1&0\end{array}\!\!
\right)\left(\!\!\!\begin{array}{c}\ups_1(1)\\\ups_2(1)\\\ups_3(1)\\\ups_4(1)\\\w_5(0)\\\w_6(0)\end{array}\!\!\!\right).
\label{s2}\end{equation}
Thus
$$
A = \widehat{(\widehat{\mb\Xi_{out}})^T\widehat{\mb\Xi_{in}}} = \left(\begin{array}{cccccc}0&0&0&0&0&0\\
0&0&0&0&0&0\\
0&0&0&0&0&0\\
0&0&0&0&0&0\\
1&1&0&0&0&1\\
0&0&1&1&1&0\end{array}
\right).
$$
Thus, there is a multi digraph $\mb\Gamma$ for which $A$ is the adjacency matrix of $L(\mb\Gamma)$.  There is no sink and to determine the structure of the sources, we observe that
$$\mb\Xi^S_{out} = \left(\begin{array}{cccc}0&1&1&0\\
1&0&0&0\\
1&1&0&0\\
0&0&1&1\end{array}\right)
\quad\text{and\;so}\quad
\mb\Xi^S = \widehat{\left(\widehat{\mb \Xi_{out}^{S}}\right)^T\widehat{\mb \Xi_{out}^{S}}} =
\left(\begin{array}{cccc}1&1&0&0\\
1&1&1&0\\
0&1&1&1\\
0&0&1&1\end{array}\right).
$$
This matrix is irreducible and thus we have one source. Therefore
$$
A^+ = \left(\begin{array}{cccccc}1&1&0&0&0&1\\
0&0&1&1&1&0\\
0&0&0&0&0&0\end{array}\right), \qquad A^- = \left(\begin{array}{cccccc}0&0&0&0&1&0\\
0&0&0&0&0&1\\
1&1&1&1&0&0\end{array}\right)
$$
and consequently
$$
A = \left(\begin{tabular}{cc|c}0&1&2\\
1&0&2\\
\hline
0&0&0\end{tabular}\right).
$$
Further,
\begin{align*}
\spp \mb a^{+,r}_1&= \{1,2,6\},\;\spp \mb a^{+,r}_2= \{3,4,5\},\\
\spp \mb a^{-,r}_1&= \{5\},\;\spp \mb a^{-,r}_2= \{6\},\;\spp \mb a^{-,r}_3= \{1,2,3,4\},
\end{align*}
hence, by \eqref{edges},
$$
(a_{12}\mapsto \{6\},a_{21}\mapsto \{5\}), \quad (a_{13}\mapsto \{1,2\},a_{31}\mapsto \emptyset),\quad (a_{23}\mapsto \{3,4\},a_{32}\mapsto \emptyset).
$$
To reconstruct $\Gamma$, we see that $\me^5$ and $\me^6$ should be combined into a single edge $\mb e$. Since, however, $J^+ = \{1,2,3,4\}, J^-=\{5,6\},$ the flow along $\me^5$ runs from 1 to 0 and hence $\mb v_1$ should correspond to 1 in the parametrization, while $\bv_2$ to 0. On the other hand, $\me^6$ runs also from 1 to 0 but from $\bv_2$ to $\bv_1$ and hence $\bv_2$ should correspond to 1, while $\bv_1$ to 0. This contradiction is in agreement with the violation of assumption \eqref{edgeid} as $(a_{12},a_{21})=(1,1)$ but in the corresponding $(k,l)= (6,5)$, both $k$ and $l$ belong to $J^-$.
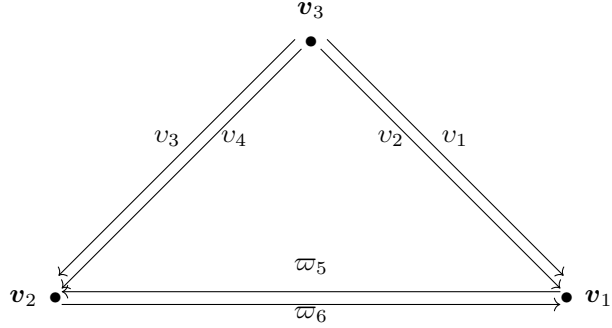
\begin{figure}
\centering
\begin{tikzpicture}[scale=0.85]
\draw[<-](-3.9,0.1) -- (3.9,0.1);
\draw[->](-3.9,-0.1) -- (3.9,-0.1);
\draw[<-](-3.95,0.35) -- (-0.25,4.05);
\draw[<-](-3.9,0.15) -- (-0.15,3.9);
\draw[<-](3.95,0.35) -- (0.25,4.05);
\draw[<-](3.9,0.15) -- (0.15,3.9);
\draw(-4,0)node{$\bullet$};
\draw(4,0)node{$\bullet$};
\draw(0,4)node{$\bullet$};
\draw(-4.5,0)node{$\bv_2$};
\draw(4.5,0)node{$\bv_1$};
\draw(0,4.5)node{$\bv_3$};
\draw(2.25,2.5) node{$\ups_1$};
\draw(1.25,2.5) node{$\ups_2$};
\draw(-2.25,2.5) node{$\ups_3$};
\draw(-1.2,2.5) node{$\ups_4$};
\draw(0,0.5) node{$\w_5$};
\draw(0,-0.25) node{$\w_6$};
  \end{tikzpicture}
\caption{The reconstructed multi digraph $\mb\Gamma$. It is seen that it cannot describe a flow on $\Gamma$ as $\w_5$ and $\w_6$ must flow in the same direction. }\label{figGam1}
\end{figure}
\begin{figure}
\centering
\begin{tikzpicture}[scale=0.85]
\draw[<-](-3.9,0.1) -- (3.9,0.1);
\draw[->](-3.9,-0.1) -- (3.9,-0.1);
\draw[<-](-3.95,0.35) -- (-0.25,4.05);
\draw[<-](-3.9,0.15) -- (-0.15,3.9);
\draw[<-](3.95,0.35) -- (0.25,4.05);
\draw[<-](3.9,0.15) -- (0.15,3.9);
\draw(-4,0)node{$\bullet$};
\draw(4,0)node{$\bullet$};
\draw(0,4)node{$\bullet$};
\draw(-4.5,0)node{$\bv_2$};
\draw(4.5,0)node{$\bv_1$};
\draw(0,4.5)node{$\bv_3$};
\draw(2.25,2.5) node{$\ups_1$};
\draw(1.25,2.5) node{$\ups_2$};
\draw(-2.25,2.5) node{$\ups_3$};
\draw(-1.2,2.5) node{$\ups_4$};
\draw(0,0.5) node{$\ups_5$};
\draw(0,-0.25) node{$\w_6$};
  \end{tikzpicture}
  \caption{The reconstructed multi digraph $\mb\Gamma$ for \eqref{s1'}, \eqref{s2'}}
  \end{figure}
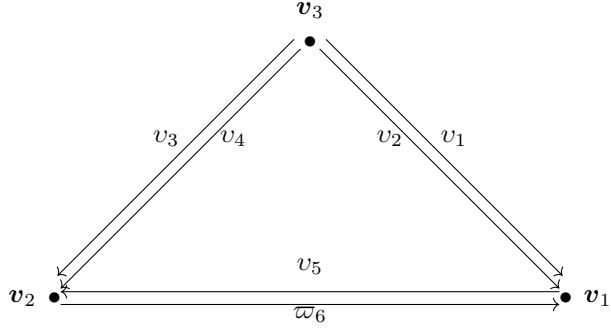

   Consider a small modification of  \eqref{s1}, \eqref{s2},
\begin{equation}\label{s1'}
\begin{split}
\p_t\ups_{j}+c_j\p_x\ups_j &=0,\quad 1\leq j\leq 5,\\
\p_t\w_{6}-c_6\p_x\w_6 &=0,
\end{split}
\end{equation}
$c_j>0$, with the last two boundary conditions of \eqref{s2} accordingly changed to
\begin{equation}\label{s2'}
\begin{split}
\ups_5(0)-\ups_1(1)-\ups_2(1)-\w_6(0)&=0,\\
\w_6(1)-\ups_3(1)-\ups_4(1)-\ups_5(1)&=0.
\end{split}
\end{equation}
The matrices $\mb\Xi_{out}, \mb\Xi_{in}, \mb\Xi^S, A^+,A^-$ and $A$ are the same as above and thus the multi digraph $\mb\Gamma$ is the same as before. However, this time on $\me^5$ we have the flow $\ups_5$, occurring from 0 to 1 and thus $\me^5$ and $\me^6$ can be combined with a parametrization  running from $0$ at $\bv_1$ to $1$ at $\bv_2$.  Assuming $c_1>c_2$, $c_3>c_4$, we identify $u^1_1=\ups_1, u^1_2=\ups_2, u^3_1=\ups_3,u^3_2=\ups_4, u^2_1 = \ups_5,u^2_2=\w^6$ and write \eqref{s1'} as a system of  $2\times 2$ hyperbolic systems on a graph $\Gamma=(\{\bv_1,\bv_2,\bv_3\}, \{\mb e_1, \mb e_2,\mb e_3\})$
\begin{equation}
\begin{array}{c}
\p_t u^1_1+c_1\p_x u^1_1 =0,\\
\p_t u^1_2+c_2\p_x u^1_2=0,
\end{array}\;\begin{array}{c} \p_t u^2_1+c_5\p_x u^2_1 =0,\\
\p_t u^2_2-c_6\p_x u^2_2=0,
\end{array}\; \begin{array}{c}\p_t u^3_1+c_3\p_x u^3_1 =0,\\\p_t u^3_2+c_4\p_x u^3_2=0,
\end{array}
\label{hip1}
\end{equation}
with boundary conditions at
\begin{equation}
\begin{tabular}{c|c}
$\bv_1:$&$u^2_1(0)-u^1_1(1)-u^1_2(1)-u^2_2(0)=0,$\\
\hline
$\bv_2:$&$u^2_2(1)-u^3_1(1)-u^3_2(1)-u^2_1(1)=0,$\\
\hline
$\bv_3:$& $\begin{array}{lc}u^1_2(0)+u^3_1(0)&=0,\\
u^1_1(0)&=0,\\
u^1_1(0)+u^1_2(0)&=0,\\
u^3_1(0)+u^3_2(0)&=0.\end{array}$
\end{tabular}\label{hipbc}
\end{equation}

\begin{figure}\center
\begin{tikzpicture}[scale=0.85]
\draw(-4,0) -- (4,0);
\draw(-4,0) -- (0,4);
\draw(0,4)--(4,0);
\draw(-4,0)node{$\bullet$};
\draw(4,0)node{$\bullet$};
\draw(0,4)node{$\bullet$};
\draw(-4.5,0)node{$\bv_2$};
\draw(4.5,0)node{$\bv_1$};
\draw(0,4.5)node{$\bv_3$};
\draw(-4,0.25)node{$1$};
\draw(-0.25,4)node{$0$};
\draw(0.25,4)node{$0$};
\draw(4,0.25)node{$1$};
\draw(4,-0.25)node{$0$};
\draw(-4,-0.25)node{$1$};
\draw(1,3.5)node{$\mb e_1$};
\draw(3,0.25)node{$\mb e_2$};
\draw(-3.5,1)node{$\mb e_3$};
\draw[->](1.5,3)--(2.5,2);
\draw[->](0.85,2.5)--(1.85,1.5);
\draw(2.5,2.5) node{$u^1_1$};
\draw(1.20,2.5) node{$u^1_2$};
\draw(-2.5,2.5) node{$u^3_1$};
\draw(-1.2,2.5) node{$u^3_2$};
\draw[->](-1.5,3)--(-2.5,2);
\draw[->](-0.85,2.5)--(-1.85,1.5);
\draw[->](0.5,0.25)--(-0.5,0.25);
\draw[<-](0.5,-0.37)--(-0.5,-0.37);
\draw(0,0.5) node{$u^2_1$};
\draw(0,-0.2) node{$u^2_2$};
  \end{tikzpicture}
   \caption{A network $\Gamma$ realizing the flow \eqref{hip1}, \eqref{hipbc}}\label{figG2}
\end{figure}
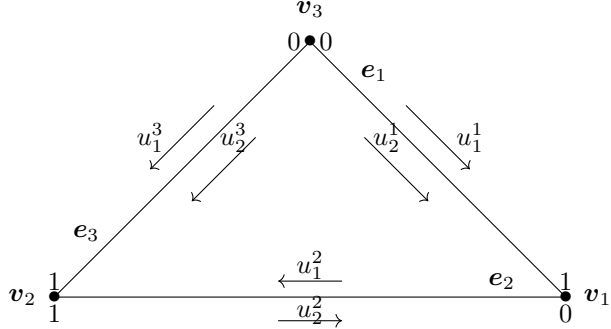

\end{example}

\appendix
\numberwithin{equation}{section}
\setcounter{equation}{0}
\section{Line digraphs}\label{appA}

Consider a digraph $G$ (possibly with multiple arcs but with no loops) and its line graph $L(G)$. For both $G$ and $L(G)$ we consider their adjacency matrices $\msf A(G)$ and $\msf A(L(G))$. The matrix $\msf A(L(G))$ is always binary, with zeroes on the diagonal. Not any binary matrix is the adjacency matrix of a line graph, see \cite{BF1,Bang}. In fact, we have
\begin{proposition}
   A binary matrix $\msf A$ is the adjacency matrix of a line digraph of a multi digraph if and only if all diagonal entries are 0 and any two columns (equivalently rows) of $\msf A$ are either equal or orthogonal (as vectors).\label{propGR}
    \end{proposition}
    For our analysis, it is important to understand the reconstruction of $G$ from a matrix $\msf A = (\msf{a}_{ij})_{1\leq i,j\leq m}$ satisfying the above conditions. As in \eqref{cr}, we write
    $$\msf A = (\msf{a}_{ij})_{1\leq i,j\leq m}=(\mb{a}^c_j)_{1\leq j\leq m} = (\mb{a}^r_i)_{1\leq i\leq m}.
    $$
         If for some $i_1$ we have $\msf a_{i_1j_1}=\ldots= \msf a_{i_1j_k} = 1,$ then it means that $\mb e_{j_1}, \ldots, \mb e_{j_k}$ join $\mb e_{i_1}$ and thus they must be incident to the same vertex $\bv$ and all  $\mb e_{i_l}$ for which $\msf a_{i_lj_1}=1$ (and thus all $\msf a_{i_lj_p}=1$ for $p=1,\ldots, k$) are outgoing from $\bv$. We further observe that all zero rows can be identified with  source(s). Similarly, zero columns correspond to sinks. If $\mb a^c_j = \mb a^r_j =0$ for some $j$, then $\mb e_j$ connects a source to a sink.

    Using the adjacency matrix of a line digraph, we cannot determine how many sources or sinks the original graph could have without additional  information. We can lump all potential sources and sinks into one source and one sink, we can have as many sinks and sources as there are zero columns and rows, respectively, or we can subdivide the arcs into some  intermediate arrangement. We describe a construction with one source and one sink and indicate its possible variants.

    We introduce $V^{in}_1 = \{r\in\{1,\ldots,m\};\; \mb a^c_r = \mb a^c_1\}$ and, inductively, $V^{in}_k = \{r\in\{1,\ldots,m\};\; \mb a^c_r = \mb a^c_{j_k}, j_k = \min\{j;\; j\notin \bigcup\limits_{1\leq p\leq k-1}  V^{in}_p\}\}$ and the process terminates at $N'$ such that $\bigcup\limits_{1\leq p\leq N'}  V^{in}_p =\{1,\ldots,m\}$. In the same way, $V^{out}_1 = \{l\in\{1,\ldots,m\};\; \mb a^r_l = \mb a^r_1\}$ and $V^{out}_k = \{l\in\{1,\ldots,m\};\; \mb a^r_l = \mb a^r_{j_k}, j_k = \min\{j;\; j\notin \bigcup\limits_{1\leq p\leq k-1}  V^{out}_p\}\}$ and the process terminates at $M'$ such that $\bigcup\limits_{1\leq p\leq M'}  V^{out}_p =\{1,\ldots,m\}$. In other words,  $\{V^{in}_j\}_{1\leq j\leq N'}$ and $\{V^{out}_i\}_{1\leq i\leq M'}$ represent the vertices of $G$ through, respectively, incoming and outgoing arcs.  If there are any zero rows in $G$, then we swap the corresponding set $V^{out}_{j_0}$ with the last set $V^{out}_{M'}$. In this way,  $V^{out}_{M'}$ represents all arcs outgoing from sources (if they exist). For this construction, we represent them as coming from a single source but other possibilities are allowed, see Remark \ref{remA24}. Similarly, if there are any zero columns, we swap the corresponding set $V^{in}_{i_0}$ with $V^{in}_{N'}$, that is, $V^{in}_{N'}$ represents the arcs incoming to sink(s).  Then we denote
    \begin{equation}
    \begin{split}
     M:= \left\{\begin{array}{lcl} M' &\text{if}& V^{out}_{M'} =\{j;\; \mb a^r_j \neq 0\},\\
     M'-1 &\text{if} & V^{out}_{M'}= \{j;\; \mb a^r_j = 0\},\end{array}\right.\\
 N:= \left\{\begin{array}{lcl} N' &\text{if}& V^{in}_{N'} =\{j;\; \mb a^c_j \neq 0\},\\
     N'-1 &\text{if} & V^{in}_{N'} =\{j;\; \mb a^c_j = 0\}.\end{array}\right.\end{split}
    \end{equation}
   Thus, we see that the number of internal (or transient) vertices, that is, which are neither sources nor sinks is $n:= M=N$. For such vertices it is important to note that, in general, $V^{out}_j$ and $V^{in}_j,$ $1\leq j\leq n,$ do not represent the same vertex. To combine $V^{out}_i$ and $V^{in}_j$ into the same vertex we have, for $1\leq i,j\leq n$,
\begin{equation}
\bv_j = \{V^{in}_j, V^{out}_i\}, \qquad a_{i_pj_r}=1\;\text{for\;some/any\;}i_p \in V^{out}_i, j_r \in V^{in}_j.
\label{vertid}
\end{equation}
With this notation, we  present a more algorithmic way of reconstructing $G$ from $\msf A$.  First, we collapse equal rows of $\msf A$ into a single row of $\msf A^+$ and equal columns of $\msf A$ into a single column and then take the transpose to get $\msf A^-$. Mathematically, let $\mbb I^+$ be a set of indices such that $\mbb I^+\cap V^{out}_i$ consists of exactly one point for each $1\leq i\leq M'$ and ordered by the order of $\{V^{out}_i\}$. Similarly, let $\mbb I^-$ be a set of indices such that $\mbb I^-\cap V^{in}_i$ consists of exactly one point for each $1\leq i\leq N'.$ We order $\mbb I^-$  consistently with $\mbb I^+$, namely, if $i_k\in \mbb I^+$ and $j_k \in \mbb I^-$ are the $k$-th indices in, respectively, $\mbb I^+$ and $\mbb I^-$, then $i_k \in V^{out}_k$ and $j_k\in V^{in}_j$ with  $j$ related to $k$ by \eqref{vertid}, that is, $\{V^{out}_k,V^{in}_j\}$ determines the same vertex $\bv_k$. As mentioned above, possible zero rows correspond to the highest indices. With this, we define
\begin{equation}
\msf A^+ = (\mb a^r_i)_{i\in \mbb I^+}, \qquad \msf A^- = \left((\mb a^c_j)_{j\in \mbb I^-}\right)^T. \label{A3}
\end{equation}
 We see now that each row of $\msf A^+$ corresponds to a vertex and each column of $\msf A^+$ corresponds to an incoming arc. If there are zero rows in $\msf A$, there is a zero row at the bottom of  $\msf A^+$ showing the presence of a (single) source. The presence of a sink is indicated by zero columns in $\msf A^+$.  Similarly, each row of $\msf A^-$ corresponds to a vertex with  arcs outgoing from it represented by nonzero entries in this row, in columns with indices corresponding to the indices of the arcs. If there are zero columns in $\msf A$, they appear as a zero row in $\msf A^-$, which represents a (single) sink. Possible sources are visible in $\msf A^-$ as zero columns.
What is important, however, is that even though we lumped all sources and sinks into one single source and a single sink, the zero columns in $\msf A^+$ and $\msf A^-$ keep track of the arcs going into the sink or out of the source, respectively. Unless there are no sources and sinks, $\msf A^+$ and $\msf A^-$ are not the incoming and outgoing incidence matrices of a graph (for the definition of these, see e.g. \cite{BF1}). Indeed, $\msf A^+$ does not contain sinks that, clearly, are part of the incoming incidence matrix. Similarly, $\msf A^-$ does not include sources. If we keep our requirement that there is only one sink and one source, then we add one row to $\msf A^+$ and one to $\msf A^-$ to represent the sink and the source, respectively. We use convention that, if both the sink and the source are present, the source is the last but one row and the sink is the last one. To determine the entries we use the required property of the incidence matrices, that there is exactly one non-zero entry in each column (expressing the fact that each arc has a unique tail and a unique head). Thus, we put 1 in the added rows in any column that was zero in $\msf A^+$ (resp. $\msf A^-$). We denote such augmented matrices by $A^+$ and $A^-$. It is easy to see that the following result is true.
\begin{proposition}
$A^+$ and $A^-$ are, respectively, incoming and outgoing incidence matrices of a multi digraph $G$ having $\msf A$ as the adjacency matrix of $L(G)$.\label{1source}
\end{proposition}
\begin{proof}  Since each column of $A^+$ and $A^-$ contains 1 only in one row, we can construct a multi digraph $G$ from them
using $\msf A(G) = A^+(A^-)^T$ as its adjacency matrix.
Since we allow $G$ to be a multi digraph, the entries of $\msf A(G)$ give the number of arcs joining the vertices. A $(k,l)$ entry in $\msf A(G)$ is given by $\mb a_k^{+,r}\cdot \mb a_l^{-,r}$ and, by construction, $\mb a_k^{+,r}$ is a row in $\msf A$ belonging to $\bv_k$ and  $\mb a_l^{-,r}$ is a column in $\msf A$ corresponding to $\bv_l$.  Nonzero entries in   $\mb a_k^{+,r}$ correspond to the arcs incoming to $\bv_k$ and nonzero entries in $\mb a_l^{-,r}$ correspond to the arcs outgoing from  $\bv_l$ so the value of $\mb a_k^{+,r}\cdot \mb a_l^{-,r}$ is the number of nonzero entries occurring at the same places in both vectors and thus the number of arcs from $\bv_l$ to $\bv_k$.

The adjacency matrix $\msf A(L(G))$ is determined as $(A^-)^T A^+$. The entries of this product are given by $\mb a_k^{-,c}\cdot \mb a_l^{+,c}$. Since each column has only one nonzero entry (equal to 1), the product will be either 0 or 1.  It is 1 if and only if there is $i$ (exactly one) such that the entry 1 appears as the $i$-th coordinate of both $\mb a_k^{-,c}$ and $\mb a_l^{+,c}.$ Now, by construction, $a_{ik}^{-}=1$ if and only if  $k \in V^{out}_i$ and $a_{il}^{+}=1$ if and only if $l \in V^{in}_j,$ where the correspondence between $j$ and $i$ is determined by \eqref{vertid}. This is equivalent to $\msf a_{kl}=1$.
\end{proof}
\begin{remark}\label{remA24}
Assume that $\msf A$ has $k$ zero rows and $l$ zero columns. We cannot identify the numbers of sinks and sources from $\msf A$ without additional information. Above, we lumped all sources and all sinks into, respectively, one source and one sink but sometimes we require more flexibility. As we know, the $k$ zero rows in $\msf A$ become $k$ zero columns in $\msf A^-$ associated with the arcs  outgoing from sources.  In a similar way, the $l$ zero columns in $\msf A$ stay to be $l$ zero columns in $\msf A^+$ associated with the arcs  incoming to sinks. We can group these arcs in an arbitrary way, with each group corresponding to, respectively a source or a sink. Assume we wish to have $\bar k$ sources and $\bar l$ sinks. Then we build the corresponding matrix $A^+$ by adding $\bar k-1$ zero rows for the sources to $\msf A^+$  and $\bar l$ rows corresponding to sinks, which will consist of zeroes everywhere apart from the columns that were zero columns in $\msf A^+$; in these columns we put 1s in such a way that each column contains only one nonzero entry (and zeroes elsewhere). Then columns having 1 in a particular row will represent the arcs incoming to a given sink.  In exactly the same way we augment $\msf A^-$, by creating $\bar l-1$ zero rows for the sinks and $\bar k$ rows for the sources.  In this way, we construct the following incoming and outgoing incidence matrices, respectively, $A^+$ and $A^-$ that, by a suitable permutation of columns, can be written as
\begin{align*}
\bar A^+ &:= A^+P = \left(\begin{tabular}{c|c|c|c} $\msf A^+_T$&$\msf A^+_S$&0&0\\
\hline
0&0&0&0\\
\hline
0&0&$\msf Z_S$&$\msf Z$\end{tabular}\right), \\
\bar A^-&:= A^-P = \left(\begin{tabular}{c|c|c|c} $\msf A^-_T$&$0$&$0$&$\msf A^-_Z$\\
\hline
0&$\msf S$&$\msf S_Z$&0\\
\hline
0&$0$&0&0\end{tabular}\right),
\end{align*}
where $P$ is the required permutation matrix and, in both cases, the first group of columns have indices corresponding to $V^{in}_i, 1\leq i\leq n$ (resp. $V^{out}_j, 1\leq j\leq n$), the second group corresponds to the arcs incoming from the sources to the transient vertices, the third group combines arcs connecting sources and sinks and the last group corresponds to the sinks fed by the transient vertices. We observe that the number of columns in each group in  $A^-$ and $A^+$ is the same. Since
$$
\bar A^+(\bar A^-)^T = (A^+P)(A^-P)^T = A^+PP^T(A^-)^T = A^+A^-,
$$
 as $P^T=P^{-1}$, see \cite[p. 140]{Mey}, for such a digraph $G$ we have
\begin{equation}
A(G) = A^+(A^-)^T = \left(\begin{tabular}{c|c|c} $\msf A^+_T(\msf A^-_T)^T$&$\msf A^+_S\msf S^T$&$0$\\
\hline
0&0&0\\
\hline
$\msf Z(\msf A_Z^-)^T$&$\msf Z_S(\msf S_Z)^T$&0\end{tabular}\right).\label{agama}
 \end{equation}
\end{remark}
\begin{example} Consider the networks $G_1$ and $G_2$ presented on Fig. \ref{figB}. We observe that grouping of sources and sinks does not affect the line graph, $L(G_1) = L(G_2),$ see Fig. \ref{lingraph}. To illustrate the discussion above, we have
\begin{equation}
\msf A = \left(\begin{array}{ccccccc}0&0&0&0&0&0&0\\
0&0&0&0&0&0&0\\
1&1&0&1&0&0&0\\
0&0&1&0&1&0&0\\
0&0&0&0&0&0&0\\
0&0&1&0&1&0&0\\
0&0&0&0&0&0&0\end{array}\right).
\label{msfa}
\end{equation}
Then
\begin{equation}
\msf A^+ = \left(\begin{array}{ccccccc}
1&1&0&1&0&0&0\\
0&0&1&0&1&0&0\\
0&0&0&0&0&0&0\end{array}\right)
\label{msfa+}
\end{equation}
and we see that there are two transient (internal) vertices $\bv_1$ and $\bv_2$ with arcs $\mb e_1,\mb e_2$ and $\mb e_4$ incoming to $\bv_1$ and arcs $\mb e_3$ and $\mb e_5$ are incoming to $\bv_2.$ The last row in $\msf A^+$ corresponds to source(s) with outgoing arcs $\mb e_1,\mb e_2,\mb e_5$ and $\mb e_7$. We also note that the zero columns in $\msf A^+$ correspond to arcs $\mb e_6$ and $\mb e_7$ that are incoming to sinks. To build $\msf A^-$, we first collapse the identical columns of $\msf A$ and take the transpose. We see from $\msf A$ that the first row of the transpose corresponds to the incoming arcs $\mb e_1,\mb e_2$ and $\mb e_4$ and thus also to vertex $\bv_1$ of $\msf A^+.$ Hence, there is no need to re-order the rows and so \eqref{A3} gives
 \begin{equation}
 \msf A^- =
 \left(\begin{array}{ccccccc}
0&0&1&0&0&0&0\\
0&0&0&1&0&1&0\\
0&0&0&0&0&0&0\end{array}\right).
 \label{msfA-}
 \end{equation}
The last row corresponds to sinks and the zero columns inform us that arcs $\mb e_1,\mb e_2, \mb e_5$ and $\mb e_7$ emanate from sources.

If we want to reconstruct the original graph with one source and one sink, then
$$
 A^+ = \left(\begin{array}{ccccccc}
1&1&0&1&0&0&0\\
0&0&1&0&1&0&0\\
0&0&0&0&0&0&0\\
0&0&0&0&0&1&1\end{array}\right),\qquad   A^- =
 \left(\begin{array}{ccccccc}
0&0&1&0&0&0&0\\
0&0&0&1&0&1&0\\
1&1&0&0&1&0&1\\
0&0&0&0&0&0&0\end{array}\right)
$$
and
$$
A^+(A^-)^T =
 \left(\begin{array}{cccc}
0&1&2&0\\
1&0&1&0\\
0&0&0&0\\
0&1&1&0\end{array}\right),
$$
which describes the right multi digraph in Fig. \ref{figB}.
\begin{figure}[h]
\center
\begin{tikzpicture}[->,>=stealth',shorten >=1pt,auto,node distance=1.5cm,
                    semithick, scale=1]
    \tikzstyle{every state}=[fill=none, draw=none, text=black]
    \node[state] (A)                 				  {$\bullet$};
  \node[state]         (B) [below of=A]     {$\bullet$};
  \node[state]         (C) [right of=A]     {$\bullet$};
\node[state]         (D) [left of=B]     {$\bullet$};
\node[state]         (E) [left of=D]     {$\bullet$};
\node[state]         (F) [above of=E]     {$\bullet$};
\node[state]         (G) [left of=E]     {$\bullet$};
\path (A) edge              node [above] {$\mb e_{7}$} (C);
\path (A) edge              node [above] {$\mb e_{1}$} (D);
\path (B) edge              node [above] {$\mb e_{2}$} (D);
\path (D) edge  [bend left]     node {$\mb e_{3}$} (E);
\path (E) edge  [bend left]     node {$\mb e_{4}$} (D);
\path (F) edge     node {$\mb e_{5}$} (E);
\path (E) edge     node [above] {$\mb e_{6}$} (G);
\node[state] (H) [left of=F] {$G_1$};
\end{tikzpicture}\qquad
\begin{tikzpicture}[->,>=stealth',shorten >=1pt,auto,node distance=1.5cm,
                    semithick, scale=1]
    \tikzstyle{every state}=[fill=none, draw=none, text=black]
  \node[state] (A)                 				  {$\bullet$};
  \node[state]         (D) [below of=A]     {$\bullet$};
\node[state]         (E) [left of=D]     {$\bullet$};
\node[state]         (G) [left of=E]     {$\bullet$};
\path (A) edge  [bend left]            node  {$\mb e_{1}$} (D);
\path (A) edge  [bend right]            node  {$\mb e_{2}$} (D);
\path (A) edge   [bend right]           node [above] {$\mb e_{5}$} (E);
\path (A) edge   [bend right]           node [above] {$\mb e_{7}$} (G);
\path (D) edge  [bend left]     node {$\mb e_{3}$} (E);
\path (E) edge  [bend left]     node {$\mb e_{4}$} (D);
\path (E) edge     node [above] {$\mb e_{6}$} (G);
\node[state] (H) [above of=G] {$G_2$};
\end{tikzpicture}
\caption{Multi digraphs $G_1$ with 3 sources and two sinks and $G_2$ with all sources and all sinks grouped into a single source and a single sink}\label{figB}
\end{figure}
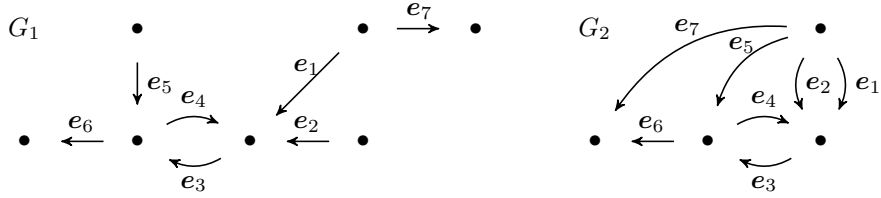
On the other hand, we can consider two sinks (maximum number, as there are two zero columns in $\msf A$) and, say, three sources. Then
$$
 A^+ = \left(\begin{array}{ccccccc}
1&1&0&1&0&0&0\\
0&0&1&0&1&0&0\\
0&0&0&0&0&0&0\\
0&0&0&0&0&0&0\\
0&0&0&0&0&0&0\\
0&0&0&0&0&0&1\\
0&0&0&0&0&1&0\end{array}\right),\qquad   A^- =
 \left(\begin{array}{ccccccc}
0&0&1&0&0&0&0\\
0&0&0&1&0&1&0\\
1&0&0&0&0&0&1\\
0&1&0&0&0&0&0\\
0&0&0&0&1&0&0\\
0&0&0&0&0&0&0\\
0&0&0&0&0&0&0\end{array}\right)
$$
and
$$
A^+(A^-)^T =
 \left(\begin{array}{ccccccc}
0&1&1&1&0&0&0\\
1&0&0&0&1&0&0\\
0&0&0&0&0&0&0\\
0&0&0&0&0&0&0\\
0&0&0&0&0&0&0\\
0&0&1&0&0&0&0\\
0&1&0&0&0&0&0\end{array}\right)
$$
which describes the left multi digraph in Fig. \ref{figB}.

It is easily seen that both digraphs have the same line digraph, shown on Fig. \ref{lingraph}, whose adjacency matrix is $\msf A$.
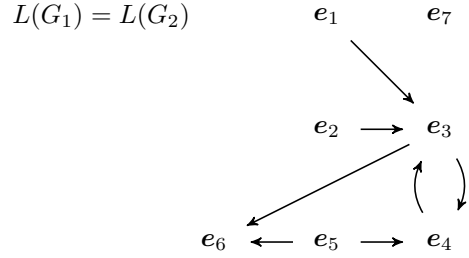
\begin{figure}[h]
\center
\begin{tikzpicture}[->,>=stealth',shorten >=1pt,auto,node distance=1.5cm,
                    semithick, scale=1]
    \tikzstyle{every state}=[fill=none, draw=none, text=black]
  \node[state] (A)                 				  {$\mb e_1$};
  \node[state] (I) [left of =A] {};
  \node[state] (H) [left of =I] {$L(G_1) = L(G_2)$};
  \node[state]         (B) [below of=A]     {$\mb e_2$};
  \node[state]         (C) [right of=A]     {$\mb e_7$};
\node[state]         (D) [below of=C]     {$\mb e_3$};
\node[state]         (E) [below of=D]     {$\mb e_4$};
\node[state]         (F) [below of=B]     {$\mb e_5$};
\node[state]         (G) [left of=F]     {$\mb e_6$};
\path (A) edge               (D);
\path (B) edge               (D);
\path (D) edge  [bend left]   (E);
\path (E) edge  [bend left]   (D);
\path (F) edge     (E);
\path (F) edge     (G);
\path (D) edge     (G);
\end{tikzpicture}
\caption{The line digraph for both $G_1$ and $G_2$}\label{lingraph}
\end{figure}

\end{example}
\bibliographystyle{abbrv}
\bibliography{JBABp2jb2}

 \end{document}